[1994/12/01]
\documentclass[10pt, reqno]{amsart}
\usepackage{a4wide}
\usepackage[english, activeacute]{babel}
\usepackage{amsmath,amsthm,amsxtra}
\usepackage{epsfig}
\usepackage{amssymb}
\usepackage{latexsym}
\usepackage{amsfonts}
\usepackage{pdftricks}
\usepackage[colorlinks=true]{hyperref}
\hypersetup{linkcolor=red,citecolor=blue,filecolor=dullmagenta,urlcolor=blue} 

\usepackage{wrapfig}
\usepackage{tikz}
\usetikzlibrary{arrows,calc,decorations.pathreplacing}
\definecolor{light-gray1}{gray}{0.90}
\definecolor{light-gray2}{gray}{0.80}
\definecolor{light-gray3}{gray}{0.60}

\title{On the dynamics of zero-speed solutions for Camassa-Holm type equations}
\author{Miguel A. Alejo}
\address{Departamento de Matem\'atica, Universidade Federal de Santa Catarina, Brasil}
\email{miguel.alejo@ufsc.br}
\thanks{M. A. was partially funded by CNPq grant no. 305205/2016-1}
\author{Manuel Fernando Cortez}
\address{Escuela Polit\'ecnica Nacional del Ecuador, Facultad de Ciencias, Departamento de Matem\'atica,
Ladr\'on de Guevara E11-253, Quito-Ecuador.}
\email{manuel.cortez@epn.edu.ec}
\thanks{}
\author{Chulkwang Kwak}
\address{Facultad de Matem\'aticas, Pontificia Universidad Cat\'olica de Chile, Campus San Joaqu\'in. Avda. Vicu\~na Mackenna 4860, Santiago, Chile}
\email{chkwak@mat.uc.cl}
\thanks{C. K. is supported by FONDECYT Postdoctorado 2017 Proyecto No. 3170067.}
\author{Claudio Mu\~noz}
\address{Departamento de Ingenier\'ia Matem\'atica DIM, and CMM UMI 2807-CNRS, Universidad de Chile, Beauchef 851 Torre Norte Piso 5, Santiago Chile}
\email{cmunoz@dim.uchile.cl}
\thanks{C.M. was partially supported by Fondecyt no. 1150202, and CMM Conicyt PIA AFB170001. Part of this work was done while the author was part of the \emph{Nonlinear Dispersive Equations} ICM Satellite Conference held in Florianopolis, Brazil (July 27-30, 2018), and he was visiting UTA and USFQ in Ambato and Quito, Ecuador (September 2018), funded by UTA, USFQ and Amarun entities. He would like to thank to the organizers for their kind hospitality and support.}
\date{\today}

\chardef\bslash=`\\ 

\newtheorem{thm}{Theorem}[section]

\newtheorem{lem}[thm]{Lemma}

\newtheorem{defn}[thm]{Definition}

\newtheorem{theorem}{Theorem}[section]

\newtheorem{corollary}[theorem]{Corollary}

\newtheorem{lemma}[theorem]{Lemma}

\newtheorem{proposition}[theorem]{Proposition}

\theoremstyle{remark}
\newtheorem{rem}{Remark}[section]
\newtheorem{remark}{Remark}[section]

\numberwithin{equation}{section}

\newcommand{\px}{\partial_x}

\newcommand{\nlop}{(1-\partial_x^2)^{-1}}

\newcommand{\R}{\mathbb{R}}

\newcommand{\la}{\lambda}

\newcommand{\ga}{\gamma}

\newcommand{\sech}{\operatorname{sech}}

\newcommand{\sgn}{\operatorname{sgn}}

\newcommand{\be}{\begin{equation}}
\newcommand{\ee}{\end{equation}}
\newcommand{\bp}{\begin{proof}}
\newcommand{\ep}{\end{proof}}
\newcommand{\bel}{\begin{equation}\label}
\newcommand{\eeq}{\end{equation}}
\newcommand{\bea}{\begin{eqnarray}}
\newcommand{\eea}{\end{eqnarray}}
\newcommand{\bee}{\begin{eqnarray*}}
\newcommand{\eee}{\end{eqnarray*}}
\newcommand{\ben}{\begin{enumerate}}
\newcommand{\een}{\end{enumerate}}

\def\bm{\left( \begin{array}{cc}}
\def\endm{\end{array}\right)}

 \providecommand{\abs}[1]{\lvert#1 \rvert}
 \providecommand{\norm}[1]{\lVert#1 \rVert}

\newcommand{\ve}{\varepsilon}

\newcommand{\ba}{\left(\begin{array}{c}}
\newcommand{\ea}{\end{array}\right)}

\newcommand{\eval}[2][\right]{\relax
  \ifx#1\right\relax \left.\fi#2#1\rvert}

\let\abs=\envert

\let\norm=\enVert

\begin{document}
\begin{abstract}
In this paper we consider globally defined solutions of Camassa-Holm (CH) type equations outside the well-known nonzero speed, peakon 
region. These equations include the standard CH and Degasperis-Procesi (DP) equations, as well as nonintegrable generalizations such as the $b$-family, elastic rod and BBM equations. Having globally defined solutions for these models, we introduce the notion of \emph{zero-speed and breather solutions}, i.e., solutions that do not decay to zero as $t\to +\infty$ on compact intervals of space. We prove that, under suitable decay assumptions, such solutions do not exist because the identically zero solution is  the global attractor of the dynamics, at least in a spatial interval of size $|x|\lesssim t^{1/2-}$ as $t\to+\infty$. As a consequence, we also show scattering and decay in CH type equations with long range nonlinearities. Our proof relies in the introduction of suitable Virial functionals \`a la Martel-Merle in the spirit of the works \cite{MP2018,MP2018_2} and \cite{KMM1} adapted to CH, DP and BBM type dynamics, one of them placed in $L^1_x$, and a second one in the energy space $H^1_x$. Both functionals combined lead to local in space decay to zero in $|x|\lesssim t^{1/2-}$ as $t\to+\infty$. Our methods do not rely on the integrable character of the equation, applying to other nonintegrable families of CH type equations as well. 
\end{abstract}
\maketitle \markboth{Decay in CH type equations} 
{Miguel A. Alejo, Manuel F. Cortez, Chulkwang Kwak and Claudio Mu\~noz}
\renewcommand{\sectionmark}[1]{}


\section{Introduction}

\subsection{Setting} Consider the Camassa-Holm equation (CH) \cite{CH} posed in the real line:
\be\label{CH}
(1-\partial_x^2)\partial_t u +3 u\partial_x u - u\partial_x^3 u -2 \partial_x u \partial_x^2 u=0, \quad (t,x)\in \R^2.
\ee
Here $u=u(t,x)\in \R$ is a scalar, real-valued function. \eqref{CH} is invariant under space-time shifts, and under the scaling $u_c(t,x) := c u(ct,x)$, $c\neq 0$. 


\medskip

The equation \eqref{CH} was originally introduced by Camassa and Holm \cite{CH} as a \emph{bi-Hamiltonian} model for waves in shallow water with peaked solutions called \emph{peakons}. It was also derived independently by Fuchssteiner and Fokas \cite{FF}, 
 in studying completely integrable generalizations of the KdV equation with bi-Hamiltonian structures obtained by using an asymptotic expansion directly in the Hamiltonian of the Euler equations in the shallow water regime. A rigorous derivation of the Camassa-Holm equation, taken from the full water waves problem, can be found in \cite{ASL} and \cite{CL}.  Since then, it has been studied by many authors in the past few years, not only due its hydrodynamical relevance (\eqref{CH} was the first equation capturing both soliton-type solitary waves as well as breaking waves), but also because of its extremely rich
mathematical structure.
%
\medskip

The CH \eqref{CH} is an \emph{integrable model} \cite{CH}, with infinitely many conservation laws. Formally, the first two are given by the $L^1$ integral
\be\label{I0}
I[u] := \int u(t,x)dx,
\ee
and the \emph{energy}
\be\label{M}
E[u] := \int (u^2 + (\partial_x u)^2)(t,x)dx.
\ee
Consequently, the standard \emph{energy space} for CH is given by the Sobolev space $H^1(\R)$. We advance that in this paper, 
we will work in the framework of global Product
solutions where both $I[u]$ and $E[u]$ are well-defined.

\medskip

The Cauchy problem for \eqref{CH} has been extensively developed in the last years. Local well-posedness (LWP) 
was established in $H^s$, $s>\frac32$ in \cite{LO,RB}.  Constantin and Escher \cite{CE2} improved this result by constructing global weak solutions in $H^1$, under the nonnegative sign condition on the Radon measure $m(u_0):=(1-\partial_x^2)u_0$. Constantin and Molinet \cite{CM} improved this result by showing continuity in time of the flow $u(t)$ in $H^1$ under same conditions on $u_0$. This framework allowed to El Dika and Molinet \cite{DM} to treat the stability problem for the multi-peakon solution using the Martel-Merle-Tsai's approach \cite{MMT}, improving the previous but fundamental single peakon stability result by Constantin and Strauss \cite{CS}. By using scalar conservation laws techniques, Bressan and Constantin \cite{BC,BC2} constructed both global conservative and global dissipative solutions for data in $H^1$. Later, Bressan, Chen and Zhang \cite{B2} showed uniqueness of these solutions. See Himonas et al.  \cite{Himonas2} for the proof of weak ill-posedness for CH in $H^1$. Finally, Linares, Ponce and Sideris, \cite{LPS} showed strong LWP in $H^1$ for data in $H^1\cap W^{1,\infty}$, a class containing peakons. The global well-posedness (GWP) of these solutions under some conditions on $m(u_0)$ seems to be an interesting open problem. See \cite{LPS,Molinet} for more details in the structure of the Cauchy problem for CH, as well as other historical developments. 

\medskip

Unlike other dispersive equations such as KdV or NLS, \eqref{CH} may develop wave breaking, in the form of a finite time unbounded slope of a bounded solution. See \cite{CE} for a fundamental first result in this direction. Regarding the blow up criteria for \eqref{CH}, sufficient conditions involve checking global quantities (usually, the $\|u_0\|_{H^{1}}$-norm), or other global conditions such as antisymmetry assumptions, or sign conditions on the associate potential $m(u_0)$ \cite{CE,CE2,CO,McKean1,Danchin}. 

\medskip
In the past few years some new sufficient conditions for wave breaking have appeared. These are the so called \textit{local-in-space blow-up criteria},  which involve only properties of the initial data in a small neighborhood of a single point. In that sense, such criteria are more general, see \cite{BCC1,BCC2,BCC3} for further details. 
 Recall that these results do not exclude the existence of global conservative and dissipative weak solutions, as explained above by \cite{BC,BC2}. Since we shall work with bounded-energy global solutions, we will discard wave breaking by assuming e.g. that $m(u_0)$ is nonnegative, or any other suitable assumption leading to global solutions. 
\medskip

Note that \eqref{CH} can be written in the more compact, conservative form \cite{DM}
\be\label{CH2}
\begin{aligned}
& \partial_t u + \partial_x \left( \frac12u^2 + (1-\partial_x^2)^{-1} \Big( u^2 + \frac12 (\partial_x u)^2 \Big) \right) =0,  \quad (t,x)\in \R\times\R, \\
&\qquad  u=u(t,x)\in\R.
\end{aligned}
\ee
This is the form of CH equation that will be used in this paper, but \emph{not the only one equation worked here}. Indeed, CH \eqref{CH2} is also part of a series of \emph{nonintegrable equations}, such as the $b-$family equation \eqref{b_family} (see  \cite{DHH}), and the \emph{Elastic Rod equation} \eqref{Elastic_Rod}, introduced in \cite{Dai}, for which in both CH is a particular case. Note however that in most cases, the above equations are \emph{not integrable}. All these general models 
can be obtained as families of asymptotically equivalent shallow water wave equations when the quadratic nonlinearity is considered, 
and a proper Kodama transformation is applied (see \cite{DP}-\cite{DGH}).

\medskip

Another model that it will be relevant for us is the Degasperis-Procesi \cite{DP,DHH}  equation
\be\label{DP}
\begin{aligned}
& \partial_t u + \frac12\partial_x \left( u^2 + 3(1-\partial_x^2)^{-1} u^2  \right) =0, \quad (t,x)\in \R\times\R, \\
&\qquad  u=u(t,x)\in\R,
\end{aligned}
\ee
obtained from the $b-$family equation \eqref{b_family} formally taking $b\uparrow 3$. This model was discovered in a search for other integrable equations, similar to the Camassa-Holm equation, and shares many similarities with CH, including peakons, bihamiltonian structure and wave breaking.
Although somehow both equations (DP and CH)  are similar in several aspects, there are very remarkable differences. For example, if we observe in detail the structure of both equations in their conservative form, we can see that the Degasperis-Procesi equation
does not involve the term $\partial_x u$ explicitly, having important consequences in the form of solitary waves, see below for more details. Additionally, some conservation laws follow different expressions to those exposed in \eqref{I0}-\eqref{Energy}.


\medskip

We will also consider CH \eqref{CH2} perturbed in the sense of the generalized BBM equation \cite{BBM} (or regularized long wave equation) posed in $\R\times \R$ (compare with \eqref{CH2} and \eqref{DP}):
\begin{equation}\label{BBM}
\begin{aligned}
&\partial_t u  + \partial_x (1- \partial_x^2)^{-1}\!\left( u + u^p \right) =0, \quad (t,x)\in \R\times\R,\\
&\qquad  u=u(t,x)\in\R, \quad p=2,3,4,\ldots
\end{aligned}
\end{equation}
This equation is a canonical shallow water model in current literature. Indeed, when $p=2$ above, \eqref{BBM} was originally derived by Benjamin, Bona and Mahony \cite{BBM} and Peregrine \cite{Peregrine}, as a model for the uni-directional propagation of long-crested, surface water waves. It also arises mathematically as a regularized version of the KdV equation, obtained by performing the standard ``Boussinesq trick''. This leads to simpler well-posedness and better dynamical properties compared with the original KdV equation. Moreover, BBM is not integrable, unlike KdV \cite{BPS,MMM}.  

\medskip
It is well-known (see \cite{BT}) that \eqref{BBM} for $p=2$ is globally well-posed in $H^s $, $s\geq 0$, and weakly ill-posed for $s<0$. As for the remaining cases $p=3,4,\ldots$, gBBM is globally well-posed in $H^1$ \cite{BBM}, thanks to the preservation of the mass and energy
\be\label{Mass}
M[u](t):=\frac12 \int \left(   u^2 + (\partial_x u)^{2} \right)(t,x)dx,
\ee
\be\label{Energy}
E[u](t):= \int \left( \frac12  u^2 +\frac{u^{p+1} }{p+1} \right)(t,x)dx.
\ee
Therefore, we identify $H^1$ as the standard \emph{energy space} for \eqref{BBM}. However, an important conservation law 
in what follows is given by the integral of the solution \eqref{I0}, which is well-defined if the solution stays in $L^1_x$.

\subsection{Solitons and peakons} As well as in the  Korteweg-de Vries regime, Camassa-Holm \eqref{CH2} describes the unidirectional propagation of waves at the surface of shallow water under the influence of gravity.  In that sense, it is a model with \emph{peakons}, explicit solutions aimed to represent sharp-crested waves in the ocean.  Following the scaling $cu(ct,x)$, $c\neq 0$, one can find basic explicit solutions. A (real-valued) {\bf peakon} is a distributional solution of \eqref{CH2} of the form \cite{CH}
\be\label{Sol}
u(t,x) = Q_c (x-ct), \quad Q_c(s) := c \,Q(s), \quad c\neq 0,
\ee
with
\[
Q (s):= e^{-|s|}.
\]
Note that peakons can move either left or right. In a fundamental paper, peakons were proved to be stable for $H^1$ perturbations by Constantin and Strauss \cite{CS},
using a sharp characterization of the energy \eqref{M}, as well as another conserved quantity for \eqref{CH2}. Explicit multi-peakons were found by Beals, Sattinger and Szmigielski \cite{BSS}. It was showed in this work that, up to time-dependent shift and scaling parameters, CH multi-peakons are just sums of peakons. Later, El-Dika and Molinet \cite{DM} showed that multi-peakons are stable, following ideas of Martel, Merle and Tsai \cite{MMT}. Very recently, Molinet \cite{Molinet} showed that CH peakons are \emph{asymptotically stable}, by proving a Liouville property in the spirit of Martel and Merle \cite{MM,MM1,MM2}. Molinet's asymptotic stability result was longer expected, but the proof is far from being direct, and it required the introduction of deep modifications to the Martel-Merle's approach. 

\medskip

Note that peakons collide elastically, and multi-peakons for CH are explicitly determined by inverse scattering techniques. Therefore, the solitonic region $|x|\gtrsim \beta |t|$, $\beta>0$, for \eqref{CH2}
is somehow well-understood from the point of the current literature. Similar results, including stable and asymptotically stable peakons, 
are available for DP \eqref{DP} \cite{LinY}-\cite{Molinet-1}.

\medskip

As for the long time asymptotics for CH, Constantin \cite{Constantin_A} described for the first time the evolution of general CH solutions in the non inviscid case (adding a term $u_x$ in \eqref{CH}), by means of inverse scattering techniques. In the inviscid CH case (i.e. \eqref{CH}) Eckhardt and Teschl \cite{ET} showed (see also \cite{Boutet}), also by using inverse scattering techniques, that sufficiently decaying solutions split into an infinite linear combination of peakons. No such results seem to hold, as far as we understand, for nonintegrable modifications.  We advance here that in this paper we shall consider part of this problem from another point of view, more related to PDE techniques. Our results will be available for CH and DP, but also for other nonintegrable models as well. 	

\medskip

The DP equation \eqref{DP} has similar peakon \cite{DP,DHH} and multipeakon solutions \cite{LS}. However, from the lack of explicit $\partial_x u$ term in \eqref{DP}, it is reasonable to see that DP admits other types of solitary waves, where $u$ (and not just $\partial_x  u$) has jumps. Indeed,  the  Degasperis-Procesi equation \eqref{DP} allows discontinuous solitons, called \textbf{shock peakons} \cite{LU} 
\begin{equation}\label{SP}
\begin{aligned}
& Q_{dp,k} := \frac{1}{t+k} \sgn(x) e^{- |x|} , \qquad k>0.
\end{aligned}
\end{equation}
The shock peakons can be also observed from the collision of the peakons (moving to the right) and anti-peakons (moving to left), see \cite{LinY,LU}. Additionally, $Q_{dp,k}$ is not in $H^1$.

\medskip

The model gBBM \eqref{BBM} is also important because it has {\bf smooth solitary waves} (see e.g. \cite{ElDika_Martel}). Indeed, for any $c>1$,
\be\label{Solitary_wave}
\begin{aligned}
u(t,x):= &~(c-1)^{\frac1{p-1}}Q_{bbm}\Bigg(\sqrt{\frac{c-1}{c}} (x-ct)\Bigg), \\
 Q_{bbm}(s):= &~\left( \frac{p+1}{2\cosh^2(\frac{p-1}{2}s)} \right)^{\frac1{p-1}},
 \end{aligned}
\ee
is a solitary wave solution of \eqref{BBM}, moving to the right with speed $c>1$. Small solitary waves in the energy space have $c\sim 1$ ($p<5$). Also, \eqref{BBM} has solitary waves with negative speed: for $c>0$ and $p$ even,
\be\label{SW2}
u(t,x):=-(c+1)^{1/(p-1)}Q_{bbm}\Bigg(\sqrt{\frac{c+1}{c}} (x+ ct)\Bigg),
\ee
is solitary wave for \eqref{BBM}, but it is never small in the energy space (see \cite{KT} for more details). The stability problem for these solitary waves it is well-known: it was studied in \cite{Bona,Weinstein,SS,BMR}. Indeed, solitary waves are stable for $p=2,3,4,5$, and stable/unstable for $p>5$, depending on the speed $c$. See also \cite{MMM,KNN} for the study of the inelastic collision problem for $p=2$.

\subsection{Main results} In this paper, we shall concentrate our efforts in the understanding of the \emph{complement of the solitonic region} for \eqref{CH}, \eqref{DP} and \eqref{BBM}. We will study two types of ``compact'' solutions, defined as follows:

\begin{defn}[Zero-speed and breather solutions]
We shall say that a \emph{nontrivial} global strong solution $u=u(t,x)$ of CH \eqref{CH2} or DP \eqref{DP} is a {\it zero-speed solution} if it satisfies
\be\label{zero_speed}
\limsup_{t\to +\infty}\|u(t)\|_{L^2(I)} >0,
\ee
for any compact set $I\subset \R$. On the other hand, if $u$ is \emph{periodic}, we shall say that $u$ is a {\it breather} solution. For the BBM case \eqref{BBM}, we shall say that a \emph{nontrivial} global strong solution $u=u(t,x)$ is a {\it speed one solution} if $u(t,\cdot -t)$ is a zero speed solution:
\be\label{one_speed}
\limsup_{t\to +\infty}\|u(t,\cdot -t)\|_{L^2(I)} >0,
\ee
for any compact set $I\subset \R$. Breather solutions are defined in a similar fashion.
\end{defn}

Note that a breather is always a zero-speed solution, but the opposite is not true in general. Important examples of breather solutions are the modified KdV \cite{Wadati,Lamb} and the Sine-Gordon \cite{Lamb} breathers. For recent works in the subject of stability, see \cite{AM,AM2,Mun,MunPal,Alejo,AMP1,AMP2}. CH peakons and BBM solitons are not zero-speed solutions.

\medskip

The purpose of this paper is to give sufficient conditions for which there are no breather nor  zero speed solutions for CH \eqref{CH2}, DP \eqref{DP}, and BBM \eqref{BBM} in the most interesting case $p=2$ (from the point of view of decay and scattering techniques), by showing decay of suitable solutions. Since the nonlinearities are quadratic, and decay in one dimension is very weak, problems \eqref{CH2}-\eqref{DP}-\eqref{BBM} enter in the framework of a \emph{supercritical scattering regime}.

\medskip

For the first result, we need the following definition. Let $b\in (0,1)$, and $I_b(t)$ the interval
\be\label{I_b}
I_b(t):= \left(- \frac{|t|^b}{\log |t|}, \frac{|t|^b}{\log |t|} \right), \qquad |t| \geq 2.
\ee
Note that $I_b$ contains any compact interval of $\R$ if $t$ is large enough. Our first result shows nonexistence of zero-speed solutions by proving \emph{decay to zero} in a time-dependent spatial interval.

\begin{thm}[Decay of globally defined solutions for CH and DP]\label{Thm2} 
Let $u=u(t,x)$ be a nontrivial global solution of \eqref{CH2} or \eqref{DP}, such that $u\in C([0,\infty), H^1(\R))\cap L^\infty(\R, L^1(\R))$. Then, 
\begin{enumerate}
\item For the solution to \eqref{CH2},
\be\label{Decay_CH}
\lim_{t\to+\infty} \|u(t)\|_{H^1(I_{1/2}(t))} =0.
\ee
\item For the solution to \eqref{DP},
\be\label{Decay_DP}
\lim_{t\to+\infty} \|u(t)\|_{L^2(I_{1/2}(t))} =0.
\ee
\end{enumerate}
In particular, $u$ cannot be a zero-speed solution in the sense of \eqref{zero_speed}.
\end{thm}

Some important remarks on this result are in order.

\begin{rem}
The hypotheses in Theorem \ref{Thm2} are satisfied in the CH case locally in time if the initial data $u_0\in H^{3/2}$ has exponential decay, see \cite{Himonas}. Note additionally that the peakon decay is the fast spatial decay allowed by the dynamics \cite{Himonas}. Also, 
global solutions in $C([0,\infty), H^1(\R))$ were showed to exist \cite{CM} if e.g. the initial data $u_0\in H^1$ satisfies $m_0:=(1-\partial_x^2)u_0\geq 0$ as a Radon measure. Also, positive peakons satisfy the hypotheses and conclusions of Theorem \ref{Thm2}.
\end{rem}

\begin{rem}
Under the assumption $u(t=0)\in L^1$, global unique entropy weak solutions to DP in the class $L^1(\R)\cap BV(\R)$ were constructed in \cite{CK}. 
\end{rem}

\begin{rem}
Theorem \ref{Thm2} does not depend on the integrability of CH (or DP), only on the conserved quantity \eqref{I0}\footnote{Sometimes, the preserved quantity $\int m(u)=\int (u-u_{xx})$ is better suited for the proofs.}, and a suitable control of the energy \eqref{Energy} (exact invariance in time is not really needed). Consequently, it is still valid (under minor modifications in the proofs coming from an adequate choice of the conservation laws) for the $b$-family equation
\be\label{b_family}
\partial_t u + \partial_x \left( \frac12u^2 + (1-\partial_x^2)^{-1} \left( \frac{b}{2}u^2 + \frac{(3-b)}2 (\partial_x u)^2 \right) \right) =0, \quad b\in (0,3),
\ee
and the elastic rod equation  
\be\label{Elastic_Rod}
\partial_t u + \partial_x \left( \frac{\ga}2u^2 + (1-\partial_x^2)^{-1} \left( \frac{(3-\ga)}{2}u^2 + \frac{\ga}2 (\partial_x u)^2 \right) \right) =0, \quad \ga\in (0,3),
\ee
provided $b,\ga\in (0,3)$. Note that $b=2$ in \eqref{b_family} represents CH and formally $b \to 3$ is the integrable Degasperis-Procesi equation \eqref{DP}. In the case of \eqref{Elastic_Rod}, $\ga=1$ represents CH. 
\end{rem}

\begin{rem}\label{Rmk1.3}
Note that the nonexistence of breathers for CH is also ensured if $xu\in L^1$. Indeed, assume that $u$ is periodic in time and nontrivial. We have
\[
\frac{d}{dt} \int xu = \frac12\int u^2 + \int (1-\partial_x^{2})^{-1} \left( u^2 + \frac12 u_x^2\right).
\]
Since  $ u^2 + \frac12 u_x^2 \geq 0$, we have from \eqref{eq:inverse op1} that $ \int (1-\partial_x^{2})^{-1} \left( u^2 + \frac12 u_x^2\right) \geq 0$. Consequently,
\[
\frac{d}{dt} \int xu \geq  \frac12\int u^2>0.
\]
Since $u$ is nontrivial, it cannot be periodic in time. A similar proof works for the DP case. See \cite{MP2018} for similar proofs in the gKdV case. However, Theorem \ref{Thm2} not only proves nonexistence of zero speed solutions, but also proves local in space decay to zero of such entities.
\end{rem}

\begin{rem}\label{MMrem}
Theorem \ref{Thm2} can be complemented with the following well-known fact: in the small data regime, global solutions of \eqref{CH} satisfy the ``exterior decay estimate''  
\[
\lim_{t\to\pm \infty} \|u(t)\|_{H^1(|x|\geq \beta |t|)} =0, \qquad \beta=\beta(\|u_0\|_{H^1}).
\]
This result can be obtained by following Martel-Merle's techniques, see e.g. \cite{MM2,KM2018}, or El Dika-Molinet paper \cite{DM}. See also \cite{KM2018} for a proof in the BBM case. Consequently, under the framework of Theorem \ref{Thm2} (subspace of the energy space), only the solitonic region $|t|^{1/2-}\ll |x| \ll |t|$ remains to be understood. See Fig. \ref{Fig:0} for further details.
\end{rem}

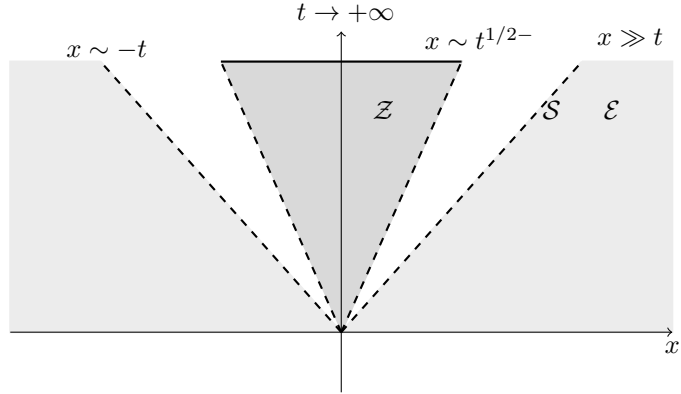
\begin{figure}[h!]
\begin{center}
\begin{tikzpicture}[scale=0.8]
\filldraw[thick, color=lightgray!60] (-2,4.5)--(2,4.5) -- (0,0);
\filldraw[thick, color=lightgray!30] (4,4.5)--(5.5,4.5) -- (5.5,0) --(0,0) -- (4,4.5);
\filldraw[thick, color=lightgray!30] (-4,4.5)--(-5.5,4.5) -- (-5.5,0) --(0,0) -- (-4,4.5);
\draw[thick,dashed] (0,0) -- (4,4.5);
\draw[thick,dashed] (0,0) -- (-4,4.5);
\draw[thick,dashed] (0,0) -- (2,4.5);
\draw[thick,dashed] (0,0) -- (-2,4.5);
\draw[thick] (-2,4.5) -- (2,4.5);
\draw[->] (-5.5,0) -- (5.5,0) node[below] {$x$};
\draw[->] (0,-1) -- (0,5) node[above] {$~t \to +\infty$};
\node at (4.8,4.9){$x\gg t$};
\node at (2.3,4.9){$x\sim t^{1/2-}$};
\node at (-3.9,4.7){$x\sim -t$};
\node at (3.5,3.7){$ \mathcal{S}$};
\node at (0.7,3.7){$ \mathcal{Z}$};
\node at (4.5,3.7){$ \mathcal{E}$};
\end{tikzpicture}
\end{center}
\caption{Graphic description of Theorem \ref{Thm2} in CH and DP cases (the only difference comes from the norms involved, $H^1$ and $L^2$ respectively). Note that all regions are symmetric wrt the $t$ axis. The set $\mathcal Z$ represents the space region where the solution converges to zero in $H^1$ or $L^2$ norm, depending on the cases CH and DP respectively. 
The thick line above $\mathcal Z$ represents $I_{1/2}(t)$ \eqref{I_b}. $\mathcal S$ stands for the solitonic region, where peakons belong (note that peakons may have any speed $c$ with $|c|>0$). Under a small data condition, the exterior region $\mathcal E$ has no mass at infinity in time, see Remark \ref{MMrem}. In the large data case, peakons dominate this region. Decay in the region $t^{1/2}\lesssim |x| \ll t$ is an open question in the general finite energy setting.}\label{Fig:0}
\end{figure}

\begin{rem}
Note that the existence of the ``zero speed'' shock peakon \eqref{SP} is in concordance with Theorem \ref{Thm2} for the DP case; indeed,
\[
\sup_{t\geq 1}\|Q_{dp,k}(t)\|_{L^1_x} <+\infty,\qquad \int_{|x|\lesssim t^{1/2-}} Q_{dp,k}^2(t,x) dx \sim \frac{1}{(t+k)^2} \to 0 \quad \hbox{as} \quad t\to +\infty.
\]
\end{rem}

{\color{black}
\begin{rem}
Theorem \ref{Thm2} can be complemented with the asymptotic results by Molinet \cite{Molinet,Molinet-1} to show stronger asymptotic stability of CH and DP peakons, not only in the solitonic region, but also inside the zero speed one $|x|\lesssim t^{1/2-}$. 
\end{rem}
}

\begin{rem}
Note that Theorem \ref{Thm2} for the DP case only shows decay for the $L^2$ norm; as for the local $H^1$ norm, we have not been able to recover a decay property. This fact remains an interesting open problem. 
\end{rem}

The proof of Theorem \ref{Thm2} is based in elementary techniques employed in \cite{MP2018} (see also \cite{MP2018_2}) for the gKdV equation $\partial_t u+ \partial_x(\partial_x^2 u +f(u))=0$, which is based in previous results for wave-like models \cite{KMM1,KMM2,KMM3,AM1,MPP}, which dealt with decay of perturbations of solutions in compact intervals of space. In this paper, to prove Theorem \ref{Thm2} along increasing in time spatial intervals, we will adapt the ideas of \cite{MP2018} to the CH and DP cases. The novelty here is the nonlocal character of CH and DP, which makes the proofs slightly different in nature; in particular, we will need some of the estimates and properties proved in \cite{KMPP2018}. There is also in CH and DP some absence of Kato smoothing properties for the second derivatives of the solution. This lack of improved decay is remedied in CH by showing complete control on the decay of the solution from the $L^1$ integral only (this is not present in KdV nor BBM, for instance). Without this control, decay estimates seem very difficult to obtain. The DP case is one example of an equation without $H^1$ decay because of this lack of smoothing properties. 

\medskip
Theorem \ref{Thm2} may appear weaker compared with other results available in the literature, but it is not clear to us whether or not better decay estimates can be proved in the supercritical scattering regime (i.e. quadratic nonlinearities). We have made no use of the integrability of the equation and no use of additional decay hypotheses for the initial data. In this last framework, see e.g. the recent work by Germain, Pusateri and Rousset \cite{GPR} for a far more precise account of the dynamics in the case of cubic KdV (mKdV), critical wrt scattering techniques, and under additional assumptions on the initial data.

\medskip

 Another purpose of this paper is to improve a recent decay result for solutions of \eqref{BBM} in the energy space, obtained in \cite{KM2018}.

\begin{thm}[Decay in gBBM \cite{KM2018}]\label{Thm1}
Let $a,b>0$ be fixed positive numbers, and let $I_{ext}(t)$ be the interval $I_{ext}(t):=\left( -\infty,  -\frac18 (1+a)t \right) \cup \left( (1+b) t ,  \infty\right)$. Consider $u_0\in H^1$ be such that, for some $\ve=\ve(b)>0$ small, one has
\be\label{Smallness}
\|u_0\|_{H^1}< \ve.
\ee
Let $u\in C(\R, H^1)$ be the corresponding global (small) solution of \eqref{BBM} with initial data $u(t=0)=u_0$. Then, there is strong decay to zero in $I_{ext}(t)$:
\be\label{Conclusion_0}
\lim_{t \to \infty}   \|u(t)\|_{H^1(I_{ext}(t))} =0.
\ee
\end{thm}

\begin{rem}
Note that this result considers the cases $p=2$ and $p=3$, which are not easy to attain using standard scattering techniques because of very weak linear decay estimates, and the presence of long range nonlinearities. Recall that the standard linear decay estimates for BMM are $O(t^{-1/3})$ \cite{Albert}. 
\end{rem}

\begin{rem}
The case of decay inside the interval $((1+b) t,+\infty)$ is probably well-known in the literature, coming from arguments similar to those exposed by El-Dika and Martel in \cite{ElDika_Martel}. However, decay for the left portion $\left( -\infty,  - at \right)$, $a>\frac18$, seems completely new as far as we understand, and it is in strong contrast with the similar decay problem for the KdV equation on the left, which has not been rigorously proved yet. 
\end{rem}

Theorem \ref{Thm1} can be regarded as a first attempt to show a complete description of decay of (suitable) solutions for \eqref{BBM} in their corresponding energy space, independently of the nonlinear character of the equation. In this paper, we shall complement Theorem \ref{Thm1} in several directions. In order to give explicit statements, we first review the current literature in the subject of decay or scattering in BBM.

\medskip

 Albert \cite{Albert} showed scattering in the $L^\infty$ norm for solutions of \eqref{BBM} provided $p>4$, with resulting global decay $O(t^{-1/3})$. Here the power 4 is important to close the nonlinear estimates, based in weighted Sobolev and Lebesgue spaces. Biler et. al \cite{BDH} showed decay in several space dimensions, using similar techniques. Hayashi and Naumkin \cite{HN} considered BBM with a difussion term, proving asymptotics for small solutions. 

\medskip

Concerning asymptotic regimes around solitary waves, the fundamental work of Miller and Weinstein \cite{MW} showed asymptotic stability of 
the BBM solitary wave in exponentially weighted Sobolev spaces. El-Dika \cite{ElDika,ElDika2} proved asymptotic stability properties of
the BBM solitary wave in the energy space. 
El-Dika and Martel \cite{ElDika_Martel} showed stability and asymptotic stability of the sum of $N$ solitary waves. 
See also Mizumachi \cite{Mizu} for similar results. All these results are proved on the right of the main part of the 
solution itself, and no information is given on the remaining left part. Theorem \ref{Thm1} is new in the sense that it also 
gives information on the left portion of the space.

%
%
%
%

\medskip

In this paper, we also prove decay for scattering supercritical BBM but outside the region $I_{ext}(t)$ considered in
Theorem \ref{Thm1}. Let $J_b(t)$ be an interval in space defined by
\begin{equation}\label{Interval J_b}
J_b(t):= \left(t - \frac{Ct^{b}}{\log t}, t + \frac{Ct^{b}}{\log t} \right),
\end{equation}
for any $C>0$, $0 \le b < 1$ and $t>2$. Note that $J_b(t)$ is centered around the line $x=t$, unlike $I_b(t)$ \eqref{I_b} in the CH and DP cases, which is centered around zero. This is mainly explained because BBM possesses a nontrivial direction for movement expressed by the fact that solitary waves move with speeds $c>1$.

\begin{thm}[Decay for BBM inside the linearly dominated region]\label{TH3}
Let $p=2$ in \eqref{BBM}. Let $u\in C(\R,H^1)\cap L^\infty(\R,L^1)$ be a solution of \eqref{BBM}, no size condition required. Then 
\be\label{Conclusion_0}
\lim_{t \to \infty}   \|u(t)\|_{H^1(J_{1/2}(t))} =0.
\ee
A similar result is valid for negative times.
\end{thm}

This results says, roughly speaking, that (small) energy solutions to BBM which stay bounded in $L^1$ need to decay to zero not only in $I_{ext}(t)$, but also inside $J_b(t).$ We do not know if this new condition is also necessary. However, the smallness is only required to ensure the validity of Theorem \ref{Thm1}. See Fig. \ref{Fig:1} for a graphic description of Theorems \ref{Thm1} and \ref{TH3}.

\begin{rem}\label{rem:Breathers BBM}
Note that the nonexistence of breathers for BBM \eqref{BBM} in the cases $p=2k$, $k=1,2,\ldots$ (even power nonlinearities), around $x=t$ is also ensured if $xu(t,\cdot-t)\in L^1$, in the same fashion as in Remark \ref{Rmk1.3}. Indeed, assume that $v(t,x):=u(t,\cdot -t)$ is periodic in time and nontrivial. Then $v$ solves $\partial_t v  + \partial_x (1- \partial_x^2)^{-1}\!\left( \partial_x^2 v + v^{2k} \right)$, $k=1,2,\ldots$ We have
\[
\frac{d}{dt} \int xv =\frac{d}{dt} \int x(v-v_{xx}) = \int v^{2k} >0 .
\]
Since $v$ is nontrivial, it cannot be periodic in time. See \cite{MP2018} for similar proofs in the gKdV case. 
\end{rem}

\begin{rem}
The nonexistence of breathers for BBM \eqref{BBM} with $p=2$ (around $x=0$, compare with Remark \ref{rem:Breathers BBM}) can also taken into account, under some particular condition of initial data. Assume either:
\begin{itemize}
\item[(1)] Positive mean condition: 
\[\int u_0 \ge 0,\]
\end{itemize}
or
\begin{itemize}
\item[(2)] Smallness and restricted negative mean condition:
\[
 - E[u_0] \le \int u_0 \le 0,
\]
and
\be\label{Smallness_BBM}
\sup_{t \in \R} \norm{u(t)}_{H^1} = \epsilon_0 < \frac{3}{2C_3},
\ee 
where $C_3$ is the Gagliardo-Nirenberg constant ($H^1 \hookrightarrow L^3$), that is, $C_3$ satisfies
\[
\norm{u}_{L^3} \le C_3 \norm{u}_{L^2}^{\frac52}\norm{u_x}_{L^2}^{\frac12}.
\]
and where $E[u_0] = E[u](t=0)$ is defined in \eqref{Energy}.
\end{itemize}

The smallness assumption \eqref{Smallness_BBM} ensures the positivity of the energy $E[u](t)$ in \eqref{Energy}. Indeed, one has
\begin{equation}\label{GN}
E[u](t) \ge \frac12\int u^2 - \frac13\int |u|^3 \ge \frac12\int u^2 - \frac{C_3 \epsilon_0}{3}\int u^2 \ge 0.
\end{equation}

\bigskip

Precisely, for a solution $u$ to \eqref{BBM} with $p=2$ and $xu\in L^1$, which is periodic in time and nontrivial, one has
\begin{equation}\label{NOB}
\frac{d}{dt} \int xu =\frac{d}{dt} \int x(u-u_{xx}) = \int u + \int u^2.
\end{equation}
Therefore, under item (1), we immediately have $\frac{d}{dt}\int xu > 0$, since $u$ is nontrivial.  

\medskip

On the other hand, using the Gagliardo-Nirenberg inequality, similarly as in \eqref{GN} under Item (2), one has
\[ \int u + \int u^2 = \int u_0 + E[u_0] + \frac12 \int u^2 - \frac13\int u^3 \ge \left(\frac12 - \frac{C_3 \epsilon_0}{3}\right) \int u^2 > 0.\]
Therefore, $u$ cannot be periodic in time.
\end{rem}

\begin{rem}
Theorem \ref{TH3} is in concordance with the existence of $L^\infty_tL^1_x$ solitary waves \eqref{Solitary_wave} with positive speeds $c>1$ for BBM, and negative speeds $c<0$ \eqref{SW2}. 
\end{rem}

\medskip

The proof of Theorem \ref{TH3} is also based on elementary virial identities placed in $L^1$ and $H^1$, with some minor but essential differences with respect to the CH and DP cases. 

\begin{figure}[h!]
\begin{center}
\begin{tikzpicture}[scale=0.8]
\filldraw[thick, color=lightgray!60] (2,4.5)--(6,4.5) -- (0,0);
\filldraw[thick, color=lightgray!30] (8,4.5)--(8,0) --(0,0) -- (8,4.5);
\filldraw[thick, color=lightgray!30] (-2,4.5)--(-4,4.5) -- (-4,0) --(0,0) -- (-2,4.5);
\draw[thick,dashed] (0,0) -- (4,4.5);
\draw[thick,dashed] (0,0) -- (6,4.5);
\draw[thick,dashed] (0,0) -- (8,4.5);
\draw[thick,dashed] (0,0) -- (2,4.5);
\draw[thick,dashed] (0,0) -- (-2,4.5);
\draw[thick] (2,4.5) -- (6,4.5);
\draw[->] (-4,0) -- (8,0) node[below] {$x$};
\draw[->] (0,-1) -- (0,5) node[above] {$~t \to +\infty$};
\node at (4.5,5){$x=t$};
\node at (7,5){$x\gg t$};
\node at (2.5,5){\small $x\sim t-t^{1/2-}$};
\node at (-2,5){$x=-t/8$};
\node at (2.5,3.7){$ \mathcal{Z}$};
\node at (4,3.7){$ \mathcal{Z}$};
\node at (5.8,3.7){$ \mathcal{S}$};
\node at (7.5,3.7){$ \mathcal{E}$};
\node at (-3,3.7){$ \mathcal{E}$};
\end{tikzpicture}
\end{center}
\caption{Graphic description of Theorems \ref{Thm1} and \ref{TH3} in the BBM case, $p=2$. Note that now regions are not symmetric wrt the $t$ axis. The set $\mathcal Z$ represents the space region of size $\sim t^{1/2-}$ where the solution converges to zero in $H^1$ norm (Theorem \ref{TH3}). The thick line above $\mathcal Z$ represents $J_{1/2}(t).$ $\mathcal S$ stands for the solitonic region, where solitons belong (small solitons have speeds $c\sim 1^+$). Under a small data condition, the exterior region $\mathcal E$ has no mass at infinity in time (Theorem \ref{Thm1}). Large solitons may exist in the white region $x<0$, moving to the left, in the cases $p$ even.}\label{Fig:1}
\end{figure}
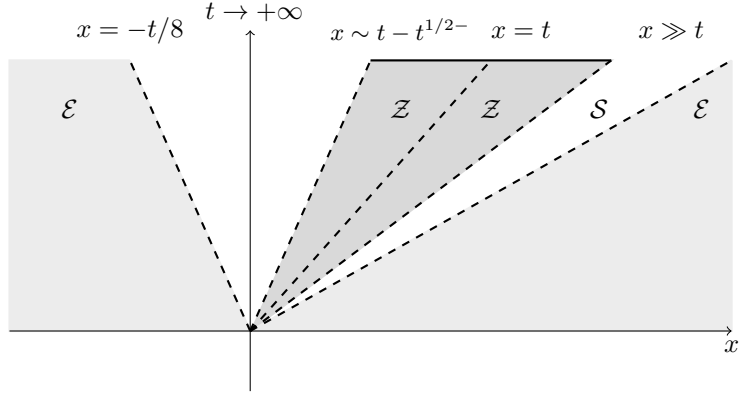

\medskip

It is worth asking whether or not the strong assumption $u\in L^\infty_t L^1_x$ can be removed in Theorems \ref{Thm2} and \ref{TH3}. This is not an easy problem, mainly because we expect important differences in the global dynamics, and much more difficult proofs. However, under mild assumptions in the growth of the $L^1_x$ norm, the global $L^\infty_t L^1_x$ condition can be removed. This approach has been successfully applied for the first time for the weak decay problem present in the Benjamin-Ono equation \cite{MP2018_2}.

\begin{thm}[Weak $H^1$-decay for CH, DP and BBM]\label{TH7}
Let $u\in C(\R,H^1)$ be a solution of \eqref{CH2}, \eqref{DP} or \eqref{BBM} $(p=2)$ 
such that
\begin{equation}\label{Condition0.1}
\int |u|(t)dx \lesssim \langle t \rangle^a.
\end{equation}
For $0 \le a < 1$, let $b := 1-a$, see Fig. \ref{Fig:3}. Then, 
\begin{enumerate}
\item For solutions to CH \eqref{CH2}, 
\be\label{H decay CH}
\liminf_{t \to \infty} \|u(t)\|_{H^1(I_b(t))}  =0.
\ee
\item For solutions to DP \eqref{DP}, 
\be\label{H decay DP}
\liminf_{t \to \infty} \|u(t)\|_{L^2(I_b(t))}  =0.
\ee
\item For solutions to BBM with $p=2$ \eqref{BBM}, 
\be\label{H decay BBM}
\liminf_{t \to \infty} \|u(t)\|_{H^1(J_b(t))}  =0.
\ee
\end{enumerate}
A similar result is valid for negative times.
\end{thm}

{\color{black}
\begin{figure}[h!]
\begin{center}
\begin{tikzpicture}[scale=0.8]
\draw[thick,dashed] (0,4) -- (4,0);
\draw[->] (-0.5,0) -- (5,0) node[below] {$a$};
\draw[->] (0,-0.5) -- (0,5) node[above] {$b$};
\node at (0,4){$\bullet$};
\node at (4,0){$\circ$};
\node at (-1.5,2){$|x|\lesssim \frac{\langle t\rangle^b}{\log t}$};
\node at (2,-1){$\int |u(t)|\lesssim \langle t\rangle^a$};
\node at (2.9,2.7){$ b=1-a$};
\end{tikzpicture}
\end{center}
\caption{Graphic description of Theorem \ref{TH7}. The horizontal axis $a$ represents the allowed rate of growth in time of the $L^1$-integral of the solution $u$, while the vertical axis the interval in $x$ where the $L^2$ (or $H^1$) norm is converging, at least along a subsequence, to zero as $t\to+\infty$ (in the BBM case, this region must be shifted by $t$). Note that the better the control on the $L^1$ norm, the larger the interval of $L^2$ decay. Depending on the value of $a$, some parts of the region below the line $b=1-a$ is also allowed, but not sharp.}\label{Fig:3}
\end{figure}
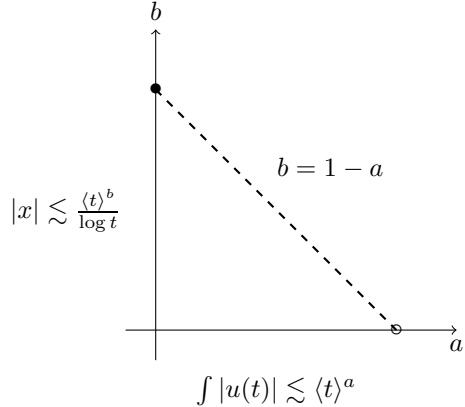
}

\begin{rem}
Theorem \ref{TH7} quantifies, at least along a sequence of times $t_n\to +\infty$, the decay to zero inside light cones $|x| \lesssim |t|^{b}$, $b=1-a$, of the solutions to CH, DP and BBM, $b$ depending on the rate 
of growth $|t|^a$ of the $L^1$ norm of the solution itself. In this sense, the better the control on $\int |u|(t)$, 
the greater the interval of decay of the solution. We conjecture that for larger intervals, these $\liminf$ are strictly positive. {\color{black} Note that, unlike BO \cite{MP2018_2}, here $a\in [0,1)$ is allowed, and not only $a\in[0,\frac12)$.}
\end{rem}

\subsection*{Organization of this paper} In Section \ref{2} 
we provide preliminary properties needed in the paper. Sections \ref{3} and \ref{4} deal 
with the CH and DP cases (Theorem \ref{Thm2}). Sections \ref{5} and \ref{6} deal with the proof of Theorem \ref{TH3}, 
the BBM case. Finally, Section \ref{sec:TH7} deals with the proof of Theorem \ref{TH7}. 

\subsection*{Acknowledgments} We thank Gustavo Ponce and Luc Molinet for some interesting remarks and 
useful comments concerning a first version of this work. 

\bigskip

\section{Preliminaries}\label{2}

%

\subsection{Canonical variables}
We provide a standard notion (canonical variable) and its fundamental 
properties in order to delicately deal with nonlocal terms (involving $(1-\px^2)^{-1}$), and introduce
some useful lemmas for estimates of nonlocal terms, which are already introduced and used in the last two author's previous works \cite{KMPP2018, KM2018}. We also refer to \cite{ElDika_Martel, ElDika2} for more details.

\medskip  
  
We say that $f$ is the canonical variable for $u$, if
\begin{equation}\label{eq:fg}
u = f-f_{xx}.
\end{equation}
Note that the canonical variable $f$ is always well-defined in $H^3$ if $u \in H^1$. 
\begin{lemma}[Equivalence of local $H^1$ norms, \cite{KMPP2018}]\label{lem:L2 comparable}
Let $f$ be as in \eqref{eq:fg}. Let $\phi$ be a smooth, bounded positive weight satisfying $|\phi''| \le \lambda \phi$ for some small but fixed $0 < \lambda \ll1$. Then, for any $a_1,a_2,a_3,a_4 > 0$, there  exist $c_1, C_1 >0$, depending on $a_j$ and $\lambda >0$, such that
\begin{equation}\label{eq:L2_est}
c_1  \int \phi \, (u^2 + u_x^2) \le \int \phi\left(a_1f^2+a_2f_x^2+a_3f_{xx}^2 +a_4 f_{xxx}^2 \right) \le C_1 \int \phi \, (u^2 + u_x^2).
\end{equation}
\end{lemma} 
\begin{remark}\label{rem:upper}
The second inequality in \eqref{eq:L2_est} still holds even when $a_i \ge 0$, $i=1,2,3,4$, while the first inequality holds only when $a_i >0$, $i=1,2,3,4$. On the other hand, one has \eqref{eq:L2_est} for $u^2$ (when $a_4 = 0$) and $u_x^2$ (when $a_1 = 0$) portions separately. See Lemmas 2.2 and 2.3 in \cite{KMPP2018}.  
\end{remark}
\begin{lemma}[Lemma 4.1 in \cite{KMPP2018}]\label{lem:Can_H1}
Let $f$ be as in \eqref{eq:fg} and $\phi$ be a smooth, bounded positive weight function. Then, one has 
\begin{equation}\label{eq:L2}
\int \phi u^2 = \int\phi\left(f^2 + 2f_x^2 + f_{xx}^2\right) - \int \phi''f^2,
\end{equation}
\begin{equation}\label{eq:H1}
\int \phi u_x^2 = \int\phi\left(f_x^2 + 2f_{xx}^2 + f_{xxx}^2\right) - \int \phi''f_x^2
\end{equation}
and
\begin{equation}\label{eq:nonlocal}
\int \phi u \nlop u = \int\phi\left(f^2 + f_x^2\right) - \frac12\int \phi''f^2.
\end{equation}
\end{lemma}

\medskip

\subsection{Comparison principle and nonlinear estimates}
The following results are essentially contained in El Dika \cite{ElDika2} and \cite{KMPP2018}. Recall that  
\be\label{def_op}
(1-\partial_x^2)^{-1} f = e^{-\frac{|x|}{2}} \ast f, \quad f \in L^2(\R).
\ee
We also have
\[
(1-\partial_x^2)^{-1}e^{-\frac{|x|}{M}} \le 2e^{-\frac{|x|}{M}}, \qquad M> 2.
\]
These two properties imply the following results.

\begin{lem}[See e.g. \cite{ElDika2,KMPP2018}]\label{lem:nonlinear}
The operator $ (1-\partial_x^2)^{-1}$ introduced in \eqref{def_op} satisfies the following properties:

\ben
\item \emph{Comparison principle.} For any $v,w\in H^1$,
\begin{equation}\label{eq:inverse op1}
v \le w \quad  \Longrightarrow \quad (1-\partial_x^2)^{-1} v \le (1-\partial_x^2)^{-1} w.
\end{equation}
\item \emph{Preservation of decay.} Suppose now that $\phi =\phi(x)$ satisfies
\begin{equation}\label{exponential weight}
\frac{C_1}{M}e^{-\frac{|x|}{M}} \le \phi \le \frac{C_2}{M}e^{-\frac{|x|}{M}},
\end{equation}
for $C_1, C_2 >0$ and $M > 2$. Then  
\begin{equation}\label{eq:inverse op2}
(1-\partial_x^2)^{-1}\phi(x) \sim \phi(x), \quad x \in \R.
\end{equation}
\item \emph{Nonlinear estimates.} One has
\begin{equation}\label{eq:nonlinear1-2}
\int\phi^{(n)} v (1-\partial_x^2)^{-1}(wh) ~\lesssim ~\norm{v}_{H^1} \int \phi(w^2 +h^2),
\end{equation}
for $\phi(x) > 0$ satisfying $|\phi^{(n)}(x)| \lesssim \phi(x)$, $n \ge 0$, and $v \in H^1$, $w,h \in L^2$.
\een
\end{lem}

%
%
%

We refer to Lemmas 2.5, 2.6 and 2.7 in \cite{KMPP2018} for various versions of \eqref{eq:nonlinear1-2}. The proof of \eqref{eq:inverse op2} in the most difficult case, namely $(1-\partial_x^2)^{-1}\phi(x) \gtrsim \phi(x)$, goes as follows. We have
\[
 \frac{C_1}{M}e^{-\frac{|x|}{M}} \leq \phi(x) \implies  \frac{C_1}{M} (1-\partial_x^2)^{-1} e^{-\frac{|x|}{M}}  \leq   (1-\partial_x^2)^{-1}\phi(x)  . 
\]
But a direct computation involving \eqref{def_op} shows 
\[(1-\partial_x^2)^{-1} e^{-\frac{|x|}{M}} = \frac{2M}{(M+2)(M-2)}\left(2Me^{-\frac{|x|}{M}} - 4 e^{-\frac{|x|}{2}} \right)\ge 2 e^{-\frac{|x|}{M}},\]
and hence $(1-\partial_x^2)^{-1} e^{-\frac{|x|}{M}} \sim e^{-\frac{|x|}{M}} $, independently of $M>2$. Consequently, for $C>0$,
\[
\phi(x) \sim \frac{C_1}{CM} e^{-\frac{|x|}{M}} \leq   (1-\partial_x^2)^{-1}\phi(x),
\]
as desired.

\bigskip

\section{Virial estimates for Camassa-Holm in $L^1$}\label{3}

\medskip

Now we start the proof of Theorem \ref{Thm2}. Let us assume, with no loss of generality, that $t\geq 2$. Let $b\in [0,1)$, and
\be\label{la}
\la(t):= \frac{t^b}{\log t}.
\ee
In the case $b=0$, we adopt the convention $\la(t)=1$. We advance that we will choose $b=\frac12$ for proving Theorem \ref{Thm2}. We have for $b>0$,
\be\label{Computations}
\begin{aligned}
\la'(t) = \frac{t^{b-1}}{\log t}\left( b- \frac1{\log t} \right), \qquad \frac{\la'(t)}{\la(t)} \lesssim \frac1t,\\
(\la'(t))^2 \lesssim \frac{t^{2(b-1)}}{\log^2 t}, \qquad (\la'(t))^3 \lesssim \frac{t^{3(b-1)}}{\log^3 t}.
\end{aligned}
\ee
Let $\phi=\tanh$, such that $\phi'=\sech^2>0$. Consider
\be\label{I}
\mathcal I(t):= \int \phi\left(\frac x{\la(t)} \right) u(t,x)dx.
\ee
Under the hypotheses of Theorem \ref{Thm2}, this functional is well-defined and bounded. We have
\begin{lem}\label{lem:dt_I}
\be\label{dt_I}
\begin{aligned}
\frac{d}{dt}\mathcal I(t) = &~   -\frac{\la'(t)}{\la(t)} \int \frac{x}{\la(t)}\phi' \left(\frac x{\la(t)} \right)u \\
&~ {} +\frac1{\la(t)}  \int  \phi' \left(\frac x{\la(t)} \right)\left( \frac12u^2 + (1-\partial_x^2)^{-1} \Big( u^2 + \frac12 u_x^2 \Big)\right).
\end{aligned}
\ee
\end{lem}
\begin{proof}
We have from \eqref{CH2},
\[ 
\begin{aligned}
\frac{d}{dt}\mathcal I(t) = &~  -\frac{\la'(t)}{\la(t)} \int \frac{x}{\la(t)}\phi' \left(\frac x{\la(t)} \right)u +  \int \phi \left(\frac x{\la(t)} \right)u_t  \\
= &~  -\frac{\la'(t)}{\la(t)} \int \frac{x}{\la(t)}\phi' \left(\frac x{\la(t)} \right)u \\
&~ {} +\frac1{\la(t)}  \int  \phi' \left(\frac x{\la(t)} \right)\left( \frac12u^2 + (1-\partial_x^2)^{-1} \Big( u^2 + \frac12 u_x^2 \Big)\right).
\end{aligned}
\]
\end{proof}

A key consequence of \eqref{dt_I} is the following estimate: for some fixed $c,~ C>0$ and $t\geq 2$,
\be\label{dt_I_lower}
\frac{d}{dt}\mathcal I(t) \geq    \frac{c}{\la(t)}  \int   \sech^2 \left(\frac x{\la(t)} \right)  \left( u^2 + u_x^2 \right) -\frac{C}{t\log^2 t},
\ee
and therefore
\be\label{Integral}
 \int_2^\infty \frac1{\la(t)}  \int   \sech^2 \left(\frac x{\la(t)} \right)  \left( u^2 + u_x^2 \right)(t,x)dxdt <+\infty,
\ee
and there exists $t_n\to +\infty$ such that  
\be\label{sucesion}
\lim_{n\to +\infty} \int   \sech^2 \left(\frac x{\la(t_n)} \right)  \left( u^2 + u_x^2 \right)(t_n,x)dx=0.
\ee 
\begin{proof}[Proof of \eqref{dt_I_lower}]
Young's inequality ensures 
\begin{equation}\label{eq:Virial0-4}
\begin{aligned}
\left|\frac{\lambda'(t)}{\lambda(t)}\int \frac{x}{\lambda(t)} \phi' \left(\frac{x}{\lambda(t)}\right) u \right| \le&~{} \frac{(\lambda'(t))^2}{2\delta_1^2} \int \left( \frac{x}{\lambda(t)} \right)^2 \phi' \left( \frac{x}{\lambda(t)} \right) \; \frac{dx}{\lambda(t)}\\
&+\frac{\delta_1^2}{2\lambda(t)}\int \phi' \left( \frac{x}{\lambda(t)} \right) u^2\\
\le&~{} \frac{(\lambda'(t))^2}{2\delta_1^2} \left(\int x^2 \sech^2(x) \right) \\
&+\frac{\delta_1^2}{2\lambda(t)}\int \phi' \left( \frac{x}{\lambda(t)} \right) u^2,
\end{aligned}
\end{equation}
for any $\delta_1 > 0$.
Choosing $\delta_1 =1$, $b=\frac12$ and using \eqref{Computations}, we get
\[
\abs{\frac{\la'(t)}{\la(t)} \int \frac{x}{\la(t)}\phi' \left(\frac x{\la(t)} \right)u} \leq  \frac{C}{t\log^2 t} +\frac{1}{2\la(t)} \int \phi' \left(\frac x{\la(t)} \right)u^2.
\]
Replacing in \eqref{dt_I} we get
\[
\frac{d}{dt}\mathcal I(t) \geq \frac1{\la(t)}  \int  \phi' \left(\frac x{\la(t)} \right) (1-\partial_x^2)^{-1} \Big( u^2 + \frac12 u_x^2 \Big) - \frac{C}{t\log^2 t}.
\]
Using the Fourier transform, we get
\[
\frac{d}{dt}\mathcal I(t) \geq \frac1{\la(t)}  \int (1-\partial_x^2)^{-1}  \phi' \left(\frac x{\la(t)} \right) \Big( u^2 + \frac12 u_x^2 \Big) - \frac{C}{t\log^2 t}.
\]
Finally, from \eqref{eq:inverse op2} we get
\[
(1-\partial_x^2)^{-1}  \sech^2\left(\frac x{\la(t)} \right) \sim \sech^2 \left(\frac x{\la(t)} \right).
\]
This ends the proof of \eqref{dt_I_lower}.
\end{proof}

\begin{remark}[The DP case]\label{rem:DP L1}
One has, similarly as Lemma \ref{lem:dt_I},
\[\begin{aligned}
\frac{d}{dt}\mathcal I(t) = &~   -\frac{\la'(t)}{\la(t)} \int \frac{x}{\la(t)}\phi' \left(\frac x{\la(t)} \right)u \\
&~ {} +\frac1{2\la(t)}  \int  \phi' \left(\frac x{\la(t)} \right)\left( u^2 + 3(1-\partial_x^2)^{-1} u^2\right)
\end{aligned}\]
for solutions to DP \eqref{DP}. An analogous argument as above ensures
\be\label{L2 DP}
\int_2^\infty \frac1{\la(t)}  \int   \sech^2 \left(\frac x{\la(t)} \right) u^2 (t,x)dxdt <+\infty,
\ee
and hence, there exists $t_n\to +\infty$ such that  
\be\label{L2 DP seq}
\lim_{n\to +\infty} \int   \sech^2 \left(\frac x{\la(t_n)} \right)  u^2 (t_n,x)dx=0.
\ee 
Note finally that no $H^1$ local decay estimate is obtained in this case.
\end{remark}

\bigskip

\section{Virial estimates in $H^1$}\label{4}

\medskip

\subsection{Energy estimates for Camassa-Holm solutions}

Consider now the weight $\phi:=\sech(4x)$ in CH \eqref{CH2}. Define
\be\label{J}
\mathcal J(t):= \int \phi\left(\frac x{\la(t)} \right) (u^2 + (\partial_x u)^2)(t,x)dx.
\ee
Note that \eqref{Decay_CH} will be proved if we show $\lim_{t\to +\infty} \mathcal J(t) =0$ .


\begin{lem}\label{lem:vari_J}
One has 
\be\label{dt_J}
\begin{aligned}
\frac{d}{dt}\mathcal J(t) = &~{}  -\frac{\la'(t)}{\la(t)} \int \frac{x}{\la(t)}\phi' \left(\frac x{\la(t)} \right)(u^2+u_x^2)\\
&~ {} +   \frac1{\la(t)} \int \phi' \left(\frac x{\la(t)} \right)uu_x^2\\
&~ {} +  \frac1{\la(t)} \int \phi' \left(\frac x{\la(t)} \right) u (1-\partial_x^2)^{-1} (2u^2 + u_x^2).
\end{aligned}
\ee
\end{lem}

\begin{remark}
The variation of the local energy $\mathcal J(t)$ as in \eqref{J} was already considered by El Dika-Molinet \cite[Lemma 3.1]{DM2007} (also in \cite[Lemma 4.2]{DM}). In \cite{DM2007, DM}, the authors claimed a differential identity on the weighted energy. However, their computation contains a small error in the procedure of the integration by parts. Nevertheless, the monotonicity property in \cite{DM2007, DM} is still valid because a leading part in their identity comes from the time derivative of the weight function.
\end{remark}

\begin{proof}[Proof of Lemma \ref{lem:vari_J}]
A direct computation of the time derivative gives
\[\begin{aligned}
\frac{d}{dt}\mathcal J(t) = &~{}  -\frac{\la'(t)}{\la(t)} \int \frac{x}{\la(t)}\phi' \left(\frac x{\la(t)} \right)(u^2+u_x^2)\\
&~{} +   2 \int \phi \left(\frac x{\la(t)} \right)(uu_t +u_xu_{xt}) \\
=:&~{} -\frac{\la'(t)}{\la(t)} \int \frac{x}{\la(t)}\phi' \left(\frac x{\la(t)} \right)(u^2+u_x^2) + J_1 + J_2.
\end{aligned}\]
Using \eqref{CH2} and performing the integration by parts, one has
\[\begin{aligned}
J_1  =&~{} \frac{2}{3\lambda(t)} \int \phi' \left(\frac x{\la(t)} \right) u^3 + \frac{1}{\lambda(t)} \int \phi' \left(\frac x{\la(t)} \right) u (1-\partial_x^2)^{-1} (2u^2 + u_x^2)\\
&~{}+ \int \phi \left(\frac x{\la(t)} \right) u_x (1-\partial_x^2)^{-1} (2u^2 + u_x^2).
\end{aligned}\]
Similarly, we obtain
\[\begin{aligned}
J_2  =&~{} \frac{1}{\lambda(t)} \int \phi' \left(\frac x{\la(t)} \right) uu_x^2 - \frac{2}{3\lambda(t)} \int \phi' \left(\frac x{\la(t)} \right) u^3 \\
&~{}- \int \phi \left(\frac x{\la(t)} \right)u_x (1-\partial_x^2)^{-1} (2u^2 + u_x^2).
\end{aligned}\]
Collecting all, one proves \eqref{dt_J}.
\end{proof}

We claim now the following estimate:
\be\label{dt_J_est}
\abs{\frac{d}{dt}\mathcal J(t)} \lesssim \frac{1}{t \log^2 t} + \frac1{\la(t)} \int \sech^4 \left(\frac x{\la(t)} \right)(u^2+u_x^2).
\ee

\begin{proof}[Proof of \eqref{dt_J_est}] 

Observe that Young's inequality ensures
\begin{equation}\label{eq:Virial01-4}
\begin{aligned}
\left| \frac{\lambda'(t)}{\lambda(t)}\int \frac{x}{\lambda(t)} \phi' \left(\frac{x}{\lambda(t)}\right) w^2 \right| \le&~{} \frac{(\lambda'(t))^2}{2\delta_2^2\lambda(t)} \left(\sup_{y \in \R}  y^2 \phi'(y) \right) \cdot \left( \int w^2 \right)\\
&+\frac{\delta_2^2}{2\lambda(t)} \int \phi' \left(\frac{x}{\lambda(t)}\right) w^2,
\end{aligned}
\end{equation}
for any $\delta_2 > 0$.
Inserting $u$ and $u_x$ into $w$ in \eqref{eq:Virial01-4} and choosing $b=\frac12$, one has
\be\label{a}
\begin{aligned}
& \abs{\frac{\la'(t)}{\la(t)} \int \frac{x}{\la(t)}\phi' \left(\frac x{\la(t)} \right)(u^2+u_x^2)}\\
 & ~ {}  \qquad \lesssim  \frac{1}{t \log^2 t}   + \frac1{\la(t)} \int \sech^4 \left(\frac x{\la(t)} \right)(u^2+u_x^2).
\end{aligned}
\ee
On the other hand, we have from the Sobolev embedding that
\be\label{b}
\abs{  \frac1{\la(t)} \int \phi' \left(\frac x{\la(t)} \right) uu_x^2} \lesssim  \frac1{\la(t)} \int \sech^4 \left(\frac x{\la(t)} \right) u_x^2.
\ee
Finally, using \eqref{eq:nonlinear1-2},
\be\label{d}
\abs{\frac1{\la(t)} \int  \phi' \left(\frac x{\la(t)} \right) u (1-\partial_x^2)^{-1} (2u^2 + u_x^2) } \lesssim \frac1{\la(t)} \int \sech^4 \left(\frac x{\la(t)} \right)(u^2 +u_x^2).
\ee
Collecting \eqref{a}, \eqref{b} 
and \eqref{d}, we obtain \eqref{dt_J_est}.
\end{proof}

\subsection{Energy estimates for Degasperis-Procesi solutions}
We introduce variants of canonical variables (compared to \eqref{eq:fg})
\begin{equation}\label{gh}
g := (4-\partial_x^2)^{-1} u \qquad \mbox{and} \qquad h:=(1-\partial_x^2)^{-1} (u^2).
\end{equation}
The Sobolev embedding ensures 
\begin{equation}\label{sobolev g}
\norm{\partial_x^j g}_{L^{\infty}} \lesssim \norm{u}_{H^1}, \qquad j=0,1,2.
\end{equation} 
Moreover, a direct calculation gives
\begin{equation}\label{L2 g}
\int (g-g_{xx})^2 \le \int u^2.
\end{equation}
Let $\mathcal K(t)$ be a localized energy functional for solutions to \eqref{DP} given by
\begin{equation}\label{K}
\mathcal K(t) := \int \phi\left(\frac x{\la(t)} \right) \left(4g^2 + 5(\partial_x g)^2 + (\partial_x^2 g)^2 \right) (t,x) dx.
\end{equation}
A modification of Lemma \ref{lem:L2 comparable} yields that the functional $\mathcal K(t)$ is well-defined when $u \in C(\R,L^2(\R))$. It is known \cite{Molinet-1} (and references therein) that the quantity
\begin{equation}\label{Energy DP}
\int (1-\partial_x^2)u (4-\partial_x^2)^{-1}u \; (t,x) \; dx = \int   (4g^2 + 5(\partial_x g)^2 + (\partial_x^2 g)^2)\; (t,x) \; dx
\end{equation}
is positive and preserved in time.

\medskip

The following lemma is the localized energy estimate for DP \eqref{DP} solutions.
\begin{lem}[Lemma 4.5 in \cite{Molinet-1}]
Let $g$ and $h$ are canonical variables defined in \eqref{gh}. Then, one has
\begin{equation}\label{dt_K}
\begin{aligned}
\frac{d}{dt} \mathcal K(t) =&~{} -\frac{\la'(t)}{\la(t)} \int \frac{x}{\la(t)}\phi' \left(\frac x{\la(t)} \right) (4g^2 + 5g_x^2 + g_{xx}^2)\\
&~{} + \frac{1}{\lambda (t)} \int \phi' \left(\frac x{\la(t)} \right) \left( \frac23 u^3 + 5 gh - 4gu^2 + g_xh_x \right).
\end{aligned}
\end{equation}
\end{lem}
We choose the weight $\phi(x) := \sech(4x)$ in \eqref{dt_K} and $b = \frac12$. Then, the claim is to show 
\begin{equation}\label{dt_K_est}
\abs{\frac{d}{dt}\mathcal K(t)} \lesssim  \frac{1}{t \log^2 t} + \frac1{\la(t)} \int \sech^4 \left(\frac x{\la(t)} \right)u^2.
\end{equation}

\begin{proof}[Proof of \eqref{dt_K_est}]
The integration by parts yields that the first term in the right-hand side of \eqref{dt_K} is written as
\[-\frac{\la'(t)}{\la(t)} \int \frac{x}{\la(t)}\phi' \left(\frac x{\la(t)} \right) (g-g_{xx})u - \frac{5\la'(t)}{2\la(t)} \int \left(  \frac{x}{\la(t)}\phi' \left(\frac x{\la(t)} \right)\right)_{xx} g^2.\]
A modification of \eqref{eq:Virial01-4} in addition to \eqref{L2 g} ensures
\[\left| \frac{\lambda'(t)}{\lambda(t)}\int \frac{x}{\lambda(t)} \phi' \left(\frac{x}{\lambda(t)}\right) (g-g_{xx})u \right| \lesssim \frac{1}{t \log^2 t} + \frac1{\la(t)} \int \sech^4 \left(\frac x{\la(t)} \right)u^2.\]
Moreover, one immediately obtains
\[\left|\frac{\la'(t)}{\la(t)} \int \left(  \frac{x}{\la(t)}\phi' \left(\frac x{\la(t)} \right)\right)_{xx} g^2 \right| \lesssim \frac{1}{t \log^2 t},\]
due to $\la'(t)/(\la(t))^3 \lesssim (t \log^2 t)^{-1}$ and $\int g^2 \le \int u^2$.

\medskip

For the rest, applying the Sobolev embedding ($H^1(\R) \hookrightarrow L^{\infty}(\R)$) and \eqref{sobolev g}, one controls
\[\left|\frac{1}{\lambda (t)} \int \phi' \left(\frac x{\la(t)} \right) \left(\frac23 u^3 - 4gu^2 \right)\right| \lesssim  \frac1{\la(t)} \int \sech^4 \left(\frac x{\la(t)} \right)u^2.\]
The integration by parts and Lemma \ref{lem:nonlinear} (3) in addition to \eqref{sobolev g} yield
\[\begin{aligned}
\left|\frac{1}{\lambda (t)} \int \phi' \left(\frac x{\la(t)} \right) \left(5gh +g_xh_x \right)\right| =&~{} \left|\frac{1}{\lambda (t)} \int \phi' \left(\frac x{\la(t)} \right) \left(5gh -g_{xx}h \right) - \frac{1}{(\lambda (t))^2} \int \phi'' \left(\frac x{\la(t)} \right) g_xh \right| \\
\lesssim&~{}\frac1{\la(t)} \int \sech^4 \left(\frac x{\la(t)} \right)u^2.
\end{aligned}\]
Collecting all, we prove \eqref{dt_K_est}.
\end{proof}

\subsection{End of proof of Theorem \ref{Thm2}}  We distinguish two cases, CH and DP.

\begin{proof}[Proof of \eqref{Decay_CH}]
From \eqref{dt_J_est}, we get after integration between $[t,t_n]$,
\[
\abs{\mathcal J(t) - \mathcal J(t_n)} \lesssim  \int_{t}^{t_n} \left(\frac{1}{s \log^2 s} + \frac1{\la(s)} \int \sech^4 \left(\frac x{\la(s)} \right)(u^2+(\partial_x u)^2)(s,x)\right)ds.
\]
Sending $n$ to infinity, using \eqref{sucesion} and \eqref{Integral}, we have $\lim_{n\to +\infty}\mathcal J(t_n)$ and
\[
\mathcal J(t)  \lesssim  \int_{t}^{\infty} \left(\frac{1}{s \log^2 s} + \frac1{\la(s)} \int \sech^4 \left(\frac x{\la(s)} \right)(u^2+(\partial_x u)^2)(s,x)\right)ds.
\]
Finally, using \eqref{Integral}, and sending $t\to +\infty$, we get \eqref{Decay_CH}.
\end{proof}

An analogous argument can be applied for the proof of \eqref{Decay_DP}.
\begin{proof}[Proof of \eqref{Decay_DP}]
Similarly, from \eqref{dt_K_est}, we get
\[
\left|\mathcal K(t) - \mathcal K(t_n)\right|  \lesssim  \int_{t}^{t_n} \left(\frac{1}{s \log^2 s} + \frac1{\la(s)} \int \sech^4 \left(\frac x{\la(s)} \right)u^2 \; (s,x)\right)ds.
\]
A direct computation gives
\[\begin{aligned}
\mathcal K(t_n) \le&~{} \int \phi \left(\frac x{\la(t)} \right) (16g^2 + 8g_x^2 + g_{xx}^2) \; (t_n,x)\\
=&~{}\int \phi \left(\frac x{\la(t)} \right) u^2 \; (t_n,x) + \frac{4}{(\la(t_n))^2} \int \phi''\left(\frac x{\la(t)} \right)g^2 \; (t_n,x),
\end{aligned}\]
which implies $\mathcal K(t_n) \to 0$ as $n \to \infty$ thanks to \eqref{L2 DP seq}. Thus, we get $\mathcal K(t) \to 0$ as $t \to \infty$. A similar computation yields
\[\begin{aligned}
\int \phi \left(\frac x{\la(t)} \right) u^2 &\le 4 \mathcal K(t) - \frac{4}{(\la(t))^2}\int \phi \left(\frac x{\la(t)} \right) g^2 \longrightarrow 0,
\end{aligned}\]
as $t \to \infty$, and thus we get \eqref{Decay_DP}.
\end{proof}
\bigskip

\section{Local virial estimates in the BBM case}\label{5} 

\medskip

The proof of Theorem \ref{TH3} in the BBM case requires other virial type estimates slightly different to those for CH. For this reason we prove them using another different $L^1$ functional. Computations are more cumbersome, but the idea behind the proof is the same as in CH.

\medskip

The change of variables $u(t,x) \mapsto u(t,x-t)$ allows us to rewrite \eqref{BBM} as follows:
\begin{equation}\label{eq:BBM}
(1- \partial_x^2)\partial_t u + \partial_x\left(\partial_x^2 u + u^p \right) =0.
\end{equation} 
This will be the equation for which we will prove the estimate in Theorem \ref{TH3}, this time around $|x|\ll |t|$.

\medskip

Consider also the following two functionals:
\begin{equation}\label{eq:I}
\mathcal{I}(t) = \int \psi \left(\frac{x}{\lambda(t)}\right) (1-\partial_x^2)u(t,x)dx,
\end{equation}
(compare with \eqref{I}) and
\begin{equation}\label{eq:J}
\mathcal{J}(t) = \frac12\int \varphi \left(\frac{x}{\lambda(t)}\right) (u^2 + (\partial_x u)^2)(t,x)dx,
\end{equation}
for certain functions $\psi, \varphi$ and $\lambda$ (to be chosen later). $\mathcal J$ is the same functional as in CH \eqref{J}, but $\mathcal I$ differs by a second derivative term. Note that the functionals $\mathcal{I}$ and $\mathcal{J}$ 
are well-defined for $u \in C(\R, H^1(\R)) \cap L^{\infty}(\R, L^1(\R))$.

\begin{lemma}
For any $t\in \R$ and $p=2$ in \eqref{eq:BBM}, we have 
\begin{equation}\label{eq:Virial0}
\begin{aligned}
\frac{d}{dt}\mathcal{I}(t) =&~{} -\frac{\lambda'(t)}{\lambda(t)}\int \frac{x}{\lambda(t)} \psi' \left(\frac{x}{\lambda(t)}\right) u\\
& + \frac{2\lambda'(t)}{(\lambda(t))^3}\int \psi''\left(\frac{x}{\lambda(t)}\right) u  + \frac{\lambda'(t)}{(\lambda(t))^3} \int \frac{x}{\lambda(t)}\psi^{(3)}\left(\frac{x}{\lambda(t)}\right) u \\
&+ \frac{1}{(\lambda(t))^3} \int \psi^{(3)}\left(\frac{x}{\lambda(t)}\right) u   + \frac{1}{\lambda(t)} \int \psi'\left(\frac{x}{\lambda(t)}\right) u^2,
\end{aligned}
\end{equation}
and
\begin{equation}\label{eq:Virial1}
\begin{aligned}
\frac{d}{dt}\mathcal{J}(t) = &~{} -\frac{\lambda'(t)}{2\lambda(t)}\int \frac{x}{\lambda(t)} \varphi' \left(\frac{x}{\lambda(t)}\right) (u^2 + u_x^2)\\
& - \frac{1}{\lambda(t)} \int \varphi' \left(\frac{x}{\lambda(t)}\right) \left(u^2 + \frac12 u_x^2 - u(1-\partial_x^2)^{-1}u \right)\\
&-\frac{1}{3\lambda(t)} \int \varphi' \left(\frac{x}{\lambda(t)}\right) u^3 \\
&+ \frac{1}{\lambda(t)} \int \varphi' \left(\frac{x}{\lambda(t)}\right) u(1-\partial^2)^{-1}(u^2).
\end{aligned}
\end{equation}
\end{lemma}
\begin{proof}
Using \eqref{BBM} and integration by parts,
\[
\begin{aligned}
\frac{d}{dt}\mathcal{I}(t) =&~{}  -\frac{\lambda'(t)}{\lambda(t)}\int \frac{x}{\lambda(t)} \psi' \left(\frac{x}{\lambda(t)}\right) (1-\partial_x^2)u \\
&~{}+ \int \psi\left(\frac{x}{\lambda(t)}\right) (1-\partial_x^2)u_t \\
=&~{} -\frac{\lambda'(t)}{\lambda(t)}\int \frac{x}{\lambda(t)} \psi' \left(\frac{x}{\lambda(t)}\right) u  + \frac{\lambda'(t)}{\lambda(t)} \int  \left(\frac{x}{\lambda(t)}\psi'\left(\frac{x}{\lambda(t)}\right)\right)_{xx} u\\
&-\int \psi\left(\frac{x}{\lambda(t)}\right) (u_{xx}+u^2)_x\\
=:&~{}-\frac{\lambda'(t)}{\lambda(t)}\int \frac{x}{\lambda(t)} \psi' \left(\frac{x}{\lambda(t)}\right) u + \mathcal I_1 + \mathcal I_2.
\end{aligned}
\]
The identity $(xf)'' = 2f' + xf''$ and a direct computation ensures that
\[
\mathcal I_1 = \frac{2\lambda'(t)}{(\lambda(t))^3}\int \psi''\left(\frac{x}{\lambda(t)}\right) u  + \frac{\lambda'(t)}{(\lambda(t))^3} \int \frac{x}{\lambda(t)}\psi^{(3)}\left(\frac{x}{\lambda(t)}\right) u.
\]
Moreover, the integration by parts yields
\[
\mathcal I_2 = \frac{1}{(\lambda(t))^3} \int \psi^{(3)}\left(\frac{x}{\lambda(t)}\right) u   + \frac{1}{\lambda(t)} \psi'\left(\frac{x}{\lambda(t)}\right) u^2.
\]
Collecting all, we prove \eqref{eq:Virial0}.

\medskip

We now focus on $\mathcal{J}(t)$ and \eqref{eq:Virial1}. We compute by the integration by parts that
\[
\begin{aligned}
\frac{d}{dt} \mathcal{J}(t) =&~{} -\frac{\lambda'(t)}{2\lambda(t)} \int \frac{x}{\lambda(t)} \varphi' \left(\frac{x}{\lambda(t)}\right) (u^2 + u_x^2)\\
&+ \int \varphi  \left(\frac{x}{\lambda(t)}\right) (uu_t + u_xu_{xt})\\
=&~{}-\frac{\lambda'(t)}{2\lambda(t)} \int \frac{x}{\lambda(t)} \varphi' \left(\frac{x}{\lambda(t)}\right) (u^2 + u_x^2)\\
& +\int \varphi  \left(\frac{x}{\lambda(t)}\right) u(1-\partial_x^2)u_t - \frac{1}{\lambda(t)}\int \varphi'\left(\frac{x}{\lambda(t)}\right) uu_{xt}\\
=:&~{} -\frac{\lambda'(t)}{2\lambda(t)} \int \frac{x}{\lambda(t)} \varphi' \left(\frac{x}{\lambda(t)}\right) (u^2 + u_x^2) + \mathcal J_1 + \mathcal J_2.
\end{aligned}
\]
Note that  \eqref{eq:BBM} can be written as follows:
\begin{equation}\label{eq:BBM1}
\partial_{t,x}^2 u = -(1-\partial_x^2)^{-1}\partial_x^2(\partial_x^2 u + u^p) = \partial_x^2 u + u + u^p - (1-\partial_x^2)^{-1}(u + u^p),
\end{equation}
and we also have the identity 
\begin{equation}\label{eq:ob}
ww_{xxx} = \frac12 (w^2)_{xxx} - \frac32(w_x^2)_x.
\end{equation}
Replacing by \eqref{eq:BBM} ($p=2$) and integrating by parts in addition to \eqref{eq:ob}, we get
\[
\begin{aligned}
\mathcal J_1 =&~{} - \int \varphi  \left(\frac{x}{\lambda(t)}\right) u (u_{xx} + u^2)_x \\ 
=&~{}  \int \varphi  \left(\frac{x}{\lambda(t)}\right) \left(\frac32(u_x^2)_x -\frac12 (u^2)_{xxx} - \frac{2}{3} (u^3)_x\right)\\
=&~{} -\frac{1}{\lambda(t)}\int \varphi'  \left(\frac{x}{\lambda(t)}\right) \left( \frac32 u_x^2 - \frac{2}{3}u^3 \right) \\
&~{} + \frac{1}{2(\lambda(t))^3} \int \varphi^{(3)}  \left(\frac{x}{\lambda(t)}\right) u^2
\end{aligned}
\]
On the other hand, we know
\begin{equation}\label{eq:ob1}
ww_{xx} = \frac12 (w^2)_{xx} - w_x^2.
\end{equation}
Replacing by \eqref{eq:BBM1} ($p=2$) and integrating by parts in addition to \eqref{eq:ob1}, we get
\[
\begin{aligned}
\mathcal J_2 =&~{} - \frac{1}{\lambda(t)}\int \varphi'\left(\frac{x}{\lambda(t)}\right) u(u_{xx} + u + u^2) \\
& ~{} + \frac{1}{\lambda(t)}\int \varphi'\left(\frac{x}{\lambda(t)}\right) u(1-\partial_x^2)^{-1}(u + u^2)\\
=&~{}  - \frac{1}{\lambda(t)}\int \varphi'\left(\frac{x}{\lambda(t)}\right) \left(\frac12(u^2)_{xx} - u_x^2 +u^2 + u^3\right) \\
&+ \frac{1}{\lambda(t)}\int \varphi'\left(\frac{x}{\lambda(t)}\right)u(1-\partial_x^2)^{-1}\left(u + u^2\right) \\
=&~{}   \frac{1}{\lambda(t)}\int \varphi'\left(\frac{x}{\lambda(t)}\right) \Big(u_x^2 - u^2 -u^3 \Big) - \frac{1}{2(\lambda(t))^3} \int \varphi^{(3)}\left(\frac{x}{\lambda(t)}\right) u^2\\
&+ \frac{1}{\lambda(t)}\int \varphi'\left(\frac{x}{\lambda(t)}\right)u(1-\partial_x^2)^{-1}\left( u + u^2 \right).
\end{aligned}
\]
Collecting all, we prove \eqref{eq:Virial1}.
\end{proof}

\subsection{Positivity of the variation of $\mathcal{J}(t)$}
We denote by $\mathcal{Q_{J}}(t)$ the quadratic term (leading term) of $\frac{d}{dt}\mathcal{J}(t)$, that is to say,
\[
\mathcal{Q_{J}}(t):= \frac{1}{\lambda(t)} \int \varphi' \left(\frac{x}{\lambda(t)}\right) \left(u^2 + \frac12 u_x^2 - u(1-\partial_x^2)^{-1}u \right).
\]

\begin{lemma}\label{lem:leading BBM}
Let $u = (1-\partial_x^2)f$ for $f \in H^3$. Then we have
\begin{equation}\label{eq:leading BBM}
\begin{aligned}
\mathcal{Q_{J}}(t) =&~{}\frac{1}{2\lambda(t)} \int \varphi' \left(\frac{x}{\lambda(t)}\right) u_x^2 + \frac{1}{\lambda(t)} \int \varphi' \left(\frac{x}{\lambda(t)}\right) (f_x^2 + f_{xx}^2)\\
&- \frac{1}{2(\lambda(t))^3} \int \varphi^{(3)} \left(\frac{x}{\lambda(t)}\right) (f^2 + f_x^2).
\end{aligned} 
\end{equation}
\end{lemma}
\begin{proof}
The proof of Lemma \ref{lem:leading BBM} can be immediately obtained from Lemma \ref{lem:Can_H1}. 
\end{proof}


\bigskip

\section{Proof of Theorem \ref{TH3}}\label{6}

\medskip

The purpose of this Section is to show Theorem \ref{TH3}. The proof is not difficult, and it is based in some new ideas introduced in \cite{MP2018} for the KdV case. In what follows, we consider a solution of this equation satisfying the hypotheses in Theorem \ref{TH3}.
\subsection{Integrability in time}
Let $C_0>0$ be an arbitrary but fixed constant. Let $\tilde J_{b}(t)$ be time-dependent space interval given by
\begin{equation}\label{J_b(t)}
\tilde J_{b}(t):= \left(- \frac{C_0 |t|^{b}}{\log |t|}, \frac{C_0|t|^{b}}{\log |t|}\right), \quad |t|\geq 2.
\end{equation}
When $b=\frac12$, $\tilde J_{\frac12}(t)$ is exactly corresponding to $J_{\frac12}(t)$ associated to \eqref{BBM} given in Theorem \ref{TH3}. We choose $\lambda(t)$ corresponding to the interval $\tilde J_{\frac12}(t)$  given by \eqref{la} ($b=\frac12$) satisfying \eqref{Computations}.

\begin{proposition}[Integrability in time of local $L^2$ norms]
Let $u$ be a solution to \eqref{BBM} such that 
\[u \in 
L^{\infty}(\R, L^1(\R)).\]
Then, there exists $2 < t_0 < \infty$\footnote{It is not necessary to find $t_0$, if $u \in L^{\infty}(\R, L^2(\R))$.} such that
\begin{equation}\label{eq:Virial0-1}
\int_{t_0}^{\infty}\frac{1}{\lambda(t)}\int \sech^2 \left(\frac{x}{\lambda(t)}\right) u^2 (t,x) \, dx \,dt < \infty.
\end{equation}
As an immediate consequence, there exists an increasing sequence of time $\{t_n\}$ $(t_n \to \infty$ as $n \to \infty)$ such that
\begin{equation}\label{eq:Virial0-2}
\int \sech^2 \left(\frac{x}{\lambda(t_n)}\right) u^2 (t_n,x) \; dx \longrightarrow 0 \mbox{ as } n \to \infty.
\end{equation}
\end{proposition}

\begin{proof}
We choose $\psi(x) = \tanh(x)$ in \eqref{eq:I}. A direct computation shows
\begin{equation}\label{eq:constant}
|\psi^{(n)}(x)| \le  n\psi'(x) = n\sech^2(x), \quad n =2,3.
\end{equation}


Recall \eqref{eq:Virial0}
\begin{equation}\label{eq:Virial0-3}
\begin{aligned}
\frac{d}{dt}\mathcal{I}(t) =&~{} -\frac{\lambda'(t)}{\lambda(t)}\int \frac{x}{\lambda(t)} \psi' \left(\frac{x}{\lambda(t)}\right) u\\
& + \frac{2\lambda'(t)}{(\lambda(t))^3}\int \psi''\left(\frac{x}{\lambda(t)}\right) u  + \frac{\lambda'(t)}{(\lambda(t))^3} \int \frac{x}{\lambda(t)}\psi^{(3)}\left(\frac{x}{\lambda(t)}\right) u \\
&+ \frac{1}{(\lambda(t))^3} \int \psi^{(3)}\left(\frac{x}{\lambda(t)}\right) u   + \frac{1}{\lambda(t)} \int \psi'\left(\frac{x}{\lambda(t)}\right) u^2.
\end{aligned}
\end{equation}
Taking $\delta_1 = \frac12$ in \eqref{eq:Virial0-4}, one has
\[
\left|-\frac{\lambda'(t)}{\lambda(t)}\int \frac{x}{\lambda(t)} \psi' \left(\frac{x}{\lambda(t)}\right) u \right| 
\le~{} 2(\lambda'(t))^2 \left(\int x^2 \sech^2(x) \right) +\frac{1}{8\lambda(t)}\int \psi' \left( \frac{x}{\lambda(t)} \right) u^2
\]
Using \eqref{eq:constant}, we similarly show for the other terms that
\[\left|\frac{2\lambda'(t)}{(\lambda(t))^3}\int \psi''\left(\frac{x}{\lambda(t)}\right) u \right| \le\frac{2(\lambda'(t))^2}{(\lambda(t))^2} \left(\int \sech^2(x) \right)  + \frac{2}{(\lambda(t))^3}\int \psi' \left( \frac{x}{\lambda(t)} \right) u^2,\]
\[
\begin{aligned}
 \left|\frac{\lambda'(t)}{(\lambda(t))^3} \int \frac{x}{\lambda(t)}\psi^{(3)}\left(\frac{x}{\lambda(t)}\right) u \right| \le &~  \frac{3(\lambda'(t))^2}{(\lambda(t))^2} \left(\int x^2 \sech^2 (x) \right)  \\
 & ~ {}  + \frac{3}{(\lambda(t))^3}\int \psi' \left( \frac{x}{\lambda(t)} \right) u^2,
 \end{aligned}
\]
and
\[\left|\frac{1}{(\lambda(t))^3} \int \psi^{(3)}\left(\frac{x}{\lambda(t)}\right) u\right| \le \frac{3(\lambda'(t))^2}{(\lambda(t))^2} \left(\int \sech^2(x) \right)  + \frac{3}{(\lambda(t))^3}\int \psi' \left( \frac{x}{\lambda(t)} \right) u^2.\]
We take a positive constant $t_0$ such that 
\[\frac{3}{(\lambda(t))^2} < \frac18, \qquad t \ge t_0.\]
Collecting all above, we have 
\begin{equation}\label{eq:Virial0-5}
 \frac{1}{2\lambda(t)} \int \psi' \left( \frac{x}{\lambda(t)} \right) u^2 \le C (\lambda'(t))^2 + \frac{d}{dt} \mathcal{I}(t), \quad t > t_0, 
 \end{equation}
for some constant $C>0$. 
Since $1/(t\log^2 t)$ is integrable on $(t_0 , \infty)$ (for $(\lambda'(t))^2$), 
we prove \eqref{eq:Virial0-1}. The standard limiting argument in addition to \eqref{eq:Virial0-1} implies \eqref{eq:Virial0-2}.
\end{proof}

\begin{remark}
In view of the proof above, the decay regime in space is determined by the integrability of $(\lambda'(t))^2$. Once a better estimate than \eqref{eq:Virial0-4} can be obtained, the interval $J_{\frac12}(t)$ would be wider. 
\end{remark}

\begin{proposition}[Integrability in time of local $H^1$ norms]
Let $u$ be a solution to \eqref{BBM} such that 
\[u \in C(\R, H^1(\R)) \cap L^{\infty}(\R, L^1(\R)).\]
Then, we have
\begin{equation}\label{eq:Virial01-1}
\int_2^{\infty}\frac{1}{\lambda(t)}\int \sech^2 \left(\frac{x}{\lambda(t)}\right) (u^2 + u_x^2) (t,x) \, dx \,dt < \infty.
\end{equation}
As an immediate consequence, there exists an increasing sequence of time $\{t_n\}$ $(t_n \to \infty$ as $n \to \infty)$ such that
\begin{equation}\label{eq:Virial01-2}
\int \sech^2 \left(\frac{x}{\lambda(t_n)}\right) (u^2 + u_x^2) (t_n,x) \; dx \longrightarrow 0 \mbox{ as } n \to \infty.
\end{equation}
\end{proposition}

\begin{proof}
It suffices to prove \eqref{eq:Virial01-1}. Recall \eqref{eq:Virial1} with $\varphi (x) = \tanh (x)$ and $p=2$.
\begin{equation}\label{eq:Virial01-3}
\begin{aligned}
\frac{d}{dt}\mathcal{J}(t) = &~{} -\frac{\lambda'(t)}{2\lambda(t)}\int \frac{x}{\lambda(t)} \varphi' \left(\frac{x}{\lambda(t)}\right) (u^2 + u_x^2)\\
& - \frac{1}{\lambda(t)} \int \varphi' \left(\frac{x}{\lambda(t)}\right) (u^2 + \frac12 u_x^2 - u(1-\partial_x^2)^{-1}u)\\
&-\frac{1}{3\lambda(t)} \int \varphi' \left(\frac{x}{\lambda(t)}\right) u^3 \\
& + \frac{1}{\lambda(t)} \int \varphi' \left(\frac{x}{\lambda(t)}\right) u(1-\partial^2)^{-1}u^2.
\end{aligned}
\end{equation}



The estimate \eqref{eq:Virial01-4} helps us to control the first term in the right-hand side of \eqref{eq:Virial01-3}. We, indeed, have by taking $\delta_2 = \frac{1}{\sqrt{2}}$ and inserting $u$ and $u_x$ into $w$ in \eqref{eq:Virial01-4} that
\[\begin{aligned}
\left| \frac{\lambda'(t)}{2\lambda(t)}\int \frac{x}{\lambda(t)} \varphi' \left(\frac{x}{\lambda(t)}\right) (u^2 + u_x^2) \right| \le&~{} \frac{(\lambda'(t))^2}{2\lambda(t)}\norm{u(t)}_{H^1}^2\left(\sup_{y \in \R}  y^2 \sech^2(y) \right)  \\
&+ \frac{1}{8\lambda(t)} \int \sech^2 \left(\frac{x}{\lambda(t)}\right) (u^2 + u_x^2).
\end{aligned}\]
Moreover, by the Sobolev embedding ($H^1(\R) \hookrightarrow L^{\infty}(\R)$) and Lemma \ref{lem:nonlinear}, we have
\[\left| \frac{1}{\lambda(t)} \int \varphi' \left(\frac{x}{\lambda(t)}\right) u^3  \right| \lesssim \frac{\norm{u(t)}_{H^1}}{\lambda(t)} \int \sech^2 \left(\frac{x}{\lambda(t)}\right) u^2\]
and
\[\left|\frac{1}{\lambda(t)} \int \varphi' \left(\frac{x}{\lambda(t)}\right) u(1-\partial^2)^{-1}u^2\right| \lesssim \frac{\norm{u}_{H^1}}{\lambda(t)} \int \sech^2 \left(\frac{x}{\lambda(t)}\right) u^2.\]
For the rest term, Lemma \ref{lem:leading BBM} and Lemma \ref{lem:L2 comparable} in addition to Remark \ref{rem:upper} yield
\[\mathcal{Q_J}(t) \ge \frac{1}{4\lambda(t)} \int \sech^2 \left(\frac{x}{\lambda(t)}\right) u_x^2 - \frac{1}{\lambda(t)} \int \sech^2 \left(\frac{x}{\lambda(t)}\right) u^2,\]
and thus, we conclude by collecting all that 
\begin{equation}\label{eq:Virial01-5}
\begin{aligned}
\frac{1}{8\lambda(t)} \int \sech^2 \left(\frac{x}{\lambda(t)}\right) u_x^2 \le &~ {}  C_1(\lambda'(t))^2 - \frac{d}{dt}\mathcal{J}(t) \\
&~ {}  + \frac{C_2}{\lambda(t)}\int \sech^2 \left(\frac{x}{\lambda(t)}\right) u^2,
\end{aligned}
\end{equation}
for some constants $C_1, C_2 > 0$ depending only on $\norm{u_0}_{H^1}$. By performing the integration in terms of $t$ on $(2,\infty)$ in addition to \eqref{eq:Virial0-1}, we have
\[\int_2^{\infty}\frac{1}{\lambda(t)} \int \sech^2 \left(\frac{x}{\lambda(t)}\right) u_x^2 < \infty.\]
Together with \eqref{eq:Virial0-5} and \eqref{eq:Virial01-5} (by considering $\frac{d}{dt} (\mathcal{I}(t) - \mathcal{J}(t))$), we prove \eqref{eq:Virial01-1}. 
\end{proof}

\subsection{End of proof of Theorem \ref{TH3}}  We will prove
\begin{proposition}[Decay of local $H^1$ norms]\label{prop:decay H1}
Let $u$ be a solution to \eqref{BBM} such that 
\[u \in C(\R, H^1(\R)) \cap L^{\infty}(\R, L^1(\R)).\]
Then, we have
\begin{equation}\label{eq:Virial1-1}
\lim_{t \to \infty}\int \sech^4 \left(\frac{x}{\lambda(t)}\right) (u^2 + u_x^2) (t,x) \; dx = 0.
\end{equation}
\end{proposition}

\begin{remark}\label{rem:proof of Thm}
Proposition \ref{prop:decay H1} immediately proves Theorem \ref{TH3}, thanks to
\[\int \sech^4 \left(\frac{x}{\lambda(t)}\right) (u^2 + u_x^2) (t,x) \; dx \gtrsim \int_{\tilde J_{\frac12}(t)} (u^2 + u_x^2) (t,x) \; dx.\]
\end{remark}

\begin{proof}[Proof of Proposition \ref{prop:decay H1}]
We take $\varphi(x) = \sech^4(x)$ in \eqref{eq:J}. Then \eqref{eq:Virial01-3} leads to the estimate
\[\left|\frac{d}{dt} \mathcal{J}(t) \right| \lesssim \frac{1}{\lambda(t)} \int \sech^2 \left(\frac{x}{\lambda(t)}\right) (u^2 + u_x^2)(t,x) \; dx + \frac{1}{t \log^2t}.\]
Taking the integral on $[t, t_n]$, for $t < t_n$ as in \eqref{eq:Virial01-2}, and \eqref{eq:Virial01-1} yield
\[
\begin{aligned}
\left|\mathcal J(t_n) - \mathcal J(t) \right| \lesssim &~ \int_t^{\infty}\!\! \frac{1}{\lambda(s)}\int \sech^2 \left(\frac{x}{\lambda(t)}\right) (u^2 + u_x^2)(s,x) \; dx ds \\
&~ {} + \int_t^{\infty} \frac{ds}{s\log^2 s} < \infty.
\end{aligned}
\]
Note that $\mathcal J(t_n) \rightarrow 0$ as $n \to \infty$, thanks to \eqref{eq:Virial01-2}. Sending $n \to \infty$ and $t \to \infty$, we complete the proof of Proposition \ref{prop:decay H1}.
\end{proof}


\bigskip

\section{Proof of Theorem \ref{TH7}}\label{sec:TH7}

\subsection{Setting}
Throughout this section, let $u\in C(\R, H^1)$ be a solution of \eqref{BBM} satisfying \eqref{Condition0.1}. We fix $0 < \iota \le 1$. For $t \ge 2$ (similarly for $t < -2$), let $\theta(t)$ be a positive smooth (at least $C^1$) function of $t$ satisfying
\begin{equation}\label{theta(t)-1}
\frac{t^a (\log t)^{1+\iota}}{\theta(t)} \le C \quad \mbox{uniformly in } t ,
\end{equation}
\begin{equation}\label{theta(t)-1.5}
\frac{\theta'(t)}{\theta(t)} \sim \frac1t
\end{equation}
and
\begin{equation}\label{theta(t)-2}
\int_2^{\infty} \frac{1}{\theta(t)\lambda(t)} = +\infty,
\end{equation}
where $\lambda(t)$ is given in \eqref{la}. Together with \eqref{theta(t)-1} and \eqref{theta(t)-1.5} in addition to \eqref{Computations}, 
one knows
\begin{equation}\label{theta(t)-3}
\int_2^{\infty} \frac{\theta'(t)t^a}{(\theta(t))^2} < +\infty \quad \mbox{and} \quad \int_2^{\infty} \frac{\lambda'(t)t^a}{\lambda(t)\theta(t)} < +\infty
\end{equation}
and
\begin{equation}\label{theta(t)-4}
\int_2^{\infty} \frac{\theta'(t)}{(\theta(t))^2} < +\infty \quad \mbox{and} \quad \int_2^{\infty} \frac{\lambda'(t)}{\theta(t)\lambda(t)} < +\infty.
\end{equation}
Note that the choice of $\theta(t) = t^a(\log t)^{1+\iota}$, $t \ge 2$, guarantees \eqref{theta(t)-1} -- \eqref{theta(t)-3}, whenever $b=1-a$, for $0 \le a < 1$.
\subsection{Refined $L^1$ virial estimates}
Let $\mathcal I_{\theta(t)}(t)$ and $\mathcal J_{\theta(t)}(t)$ be functionals defined by
\[\mathcal I_{\theta(t)}(t) := \frac{1}{\theta(t)}\mathcal I(t)\]
and
\[\mathcal J_{\theta(t)}(t) := \frac{1}{\theta(t)}\mathcal J(t),\]
where the functionals $\mathcal I(t)$ and $\mathcal J(t)$ are given in \eqref{eq:I} and \eqref{eq:J}, respectively. Note that $\sup_{t \in \R} \mathcal I_{\theta(t)}(t) \le C$, thanks to \eqref{theta(t)-1}. Moreover, 
we already know $\mathcal J(t) < \infty$ uniformly in time, when $u \in C(\R, H^1(\R))$.

\begin{lemma}
For any $t\in \R$ and $p=2$ in \eqref{BBM}, we have
\begin{equation}\label{eq:Virial0 general}
\begin{aligned}
\frac{d}{dt}\mathcal{I}_{\theta(t)}(t) =&~{} -\frac{\theta'(t)}{(\theta(t))^2} \int \psi\left(\frac{x}{\lambda(t)}\right) (1-\px^2)u \\
&  -\frac{\lambda'(t)}{\theta(t)\lambda(t)}\int \frac{x}{\lambda(t)} \psi' \left(\frac{x}{\lambda(t)}\right) u\\
& + \frac{2\lambda'(t)}{\theta(t)(\lambda(t))^3}\int \psi''\left(\frac{x}{\lambda(t)}\right) u  + \frac{\lambda'(t)}{\theta(t)(\lambda(t))^3} \int \frac{x}{\lambda(t)}\psi^{(3)}\left(\frac{x}{\lambda(t)}\right) u \\
&+ \frac{1}{\theta(t)(\lambda(t))^3} \int \psi^{(3)}\left(\frac{x}{\lambda(t)}\right) u   \\
& + \frac{1}{\theta(t)\lambda(t)} \int \psi'\left(\frac{x}{\lambda(t)}\right) u^2
\end{aligned}
\end{equation}
and
\begin{equation}\label{eq:Virial1 general}
\begin{aligned}
\frac{d}{dt}\mathcal{J}_{\theta(t)}(t) = &~{}-\frac{\theta'(t)}{2(\theta(t))^2} \int \varphi\left(\frac{x}{\lambda(t)}\right) \left( u^2 + u_x^2\right) \\
&-\frac{\lambda'(t)}{2\theta(t)\lambda(t)}\int \frac{x}{\lambda(t)} \varphi' \left(\frac{x}{\lambda(t)}\right) (u^2 + u_x^2)\\
& - \frac{1}{\theta(t)\lambda(t)} \int \varphi' \left(\frac{x}{\lambda(t)}\right) \left(u^2 + \frac12 u_x^2 - u(1-\partial_x^2)^{-1}u \right)\\
&-\frac{1}{3\theta(t)\lambda(t)} \int \varphi' \left(\frac{x}{\lambda(t)}\right) u^3 \\
&+ \frac{1}{\theta(t)\lambda(t)} \int \varphi' \left(\frac{x}{\lambda(t)}\right) u(1-\partial^2)^{-1}(u^2)
\end{aligned}
\end{equation}
\end{lemma}

\begin{proof}
The proof of \eqref{eq:Virial0 general} immediately follows from 
\[\frac{d}{dt}\mathcal{I}_{\theta(t)}(t) = -\frac{\theta'(t)}{(\theta(t))^2} \int \psi\left(\frac{x}{\lambda(t)}\right) (1-\px^2)u + \frac{1}{\theta(t)}\frac{d}{dt} \mathcal I(t)\]
and \eqref{eq:Virial0}. Similarly, we have \eqref{eq:Virial1 general} from
\[\frac{d}{dt}\mathcal{J}_{\theta(t)}(t) = -\frac{\theta'(t)}{2(\theta(t))^2} \int \varphi\left(\frac{x}{\lambda(t)}\right) \left( u^2 + u_x^2\right) + \frac{1}{\theta(t)}\frac{d}{dt} \mathcal J(t)\]
and \eqref{eq:Virial1}.
\end{proof}

\subsection{Proof of Theorem \ref{TH7}}
The proofs of Theorem \ref{TH7} for CH, DP and BBM cases are almost identical. We only give the detailed proof for BBM case here, and proofs for the others will be commented at the end of this section.
\subsubsection{The BBM case} Here we have
\begin{proposition}[Time-integrability of local $L^2$]\label{prop:integrability of u general}
We fix $p=2$. Let $u$ be a solution to \eqref{BBM} such that 
\[u \in C(\R, H^1(\R)).\]
Moreover, $u$ satisfies 
\[\int |u| \lesssim \langle t \rangle^a.\]
Let $b = 1-a$ for $0 \le a<1$. Then, we have
\begin{equation}\label{eq:L^2 general0}
\int_{t_0}^{\infty}\frac{1}{\theta(t)\lambda(t)}\int \sech^2 \left(\frac{x}{\lambda(t)}\right) u^2 (t,x) \, dx \,dt < \infty.
\end{equation}
\end{proposition}

\begin{proof}
We choose $\psi = \tanh$ in \eqref{eq:Virial0 general}. Using \eqref{theta(t)-3} and $1/(\lambda(t))^2 \lesssim 1$ in \eqref{eq:Virial0 general}, one knows the first four terms in the right-hand side of \eqref{eq:Virial0 general} are bounded by a integrable (in time) function (denoted by $\Omega(t)$). Moreover, we have
\[
\begin{aligned}
\left|\frac{1}{\theta(t)(\lambda(t))^3} \int \psi^{(3)}\left(\frac{x}{\lambda(t)}\right) u\right| \le &~{} \frac{1}{2\theta(t)\lambda(t)} \int \sech^2\left(\frac{x}{\lambda(t)}\right) u^2 \\
& ~ {}+ \frac{1}{2\theta(t)(\lambda(t))^4} \int \sech^2 x \;dx .
\end{aligned}
\]
Note that the condition $b = 1-a$ for $0\le a < 1$ in $\lambda(t)$ ensures the integrability of $1/(\theta(t)(\lambda(t))^4)$, precisely,
\[ \int _2^{\infty} \frac{1}{\theta(t)(\lambda(t))^4} \le  \int_2^{\infty} \frac{C(\log t)^{3-\iota}}{t^{4-3a}} < \infty \; \Longleftrightarrow \; a < 1.\]
Thus, we have 
\[\frac{d}{dt}\mathcal{I}_{\theta(t)}(t)  \ge  \frac{1}{2\theta(t)\lambda(t)} \int \psi'\left(\frac{x}{\lambda(t)}\right) u^2 - \Omega(t),\]
which implies (passing to the standard argument) \eqref{eq:L^2 general0}.
\end{proof}

\begin{proposition}[Time-integrability of local $\dot{H}^1$]\label{prop:integrability of u_x general}
We fix $p=2$. Let $u$ be a solution to \eqref{BBM} such that 
\[u \in C(\R, H^1(\R)).\]
Moreover, $u$ satisfies 
\[\int |u| \lesssim \langle t \rangle^a.\]
Let $b = 1-a$ for $0 \le a<1$. Then, we have
\begin{equation}\label{eq:H^1 general0}
\int_{t_0}^{\infty}\frac{1}{\theta(t)\lambda(t)}\int \sech^2 \left(\frac{x}{\lambda(t)}\right) u_x^2 (t,x) \, dx \,dt < \infty.
\end{equation}
\end{proposition}

\begin{proof}
We choose $\varphi = \tanh$ in \eqref{eq:Virial1 general}. Using \eqref{theta(t)-4} and $u \in C(\R, H^1)$, one controls the first two terms in the right-hand side of \eqref{eq:Virial0 general} by a integrable (in time) function (denoted by $\widetilde \Omega(t)$). Lemma \ref{lem:leading BBM} helps us to deal with the third term, i.e.,
\[
\begin{aligned}
\frac{1}{\theta(t)\lambda(t)}& \int \varphi' \left(\frac{x}{\lambda(t)}\right) \left(u^2 + \frac12 u_x^2 - u(1-\partial_x^2)^{-1}u \right) \\
=&~{} \frac{1}{2\theta(t)\lambda(t)} \int \varphi' \left(\frac{x}{\lambda(t)}\right) u_x^2  + \frac{1}{\theta(t)\lambda(t)} \int \varphi' \left(\frac{x}{\lambda(t)}\right) (f_x^2 + f_{xx}^2)\\
& - \frac{1}{2\theta(t)(\lambda(t))^3} \int \varphi^{(3)} \left(\frac{x}{\lambda(t)}\right) (f^2 + f_x^2).
\end{aligned}
\]
Note that the last term can be controlled by
\[\frac{1}{2\theta(t)\lambda(t)} \int \sech^2\left(\frac{x}{\lambda(t)}\right) u^2 + \frac{\left( \int \sech^2 (x) \;dx \right) \norm{u}_{H^1}^2}{2\theta(t)(\lambda(t))^4},\]
where the latter term is integrable in time. For the rest terms in \eqref{eq:Virial1 general}, we use the Sobolev embedding and Lemma \ref{lem:nonlinear} to control them by
\[\frac{1}{\theta(t)\lambda(t)} \int \sech^2\left(\frac{x}{\lambda(t)}\right) u^2.\]
Thus, we have
\[
\begin{aligned}
-\frac{d}{dt}\mathcal{J}_{\theta(t)}(t) \ge &~ {}  \frac{1}{2\theta(t)\lambda(t)} \int \varphi' \left(\frac{x}{\lambda(t)}\right) u_x^2 \\
 &~ {}  - \frac{C^*}{\theta(t)\lambda(t)} \int \sech^2\left(\frac{x}{\lambda(t)}\right) u^2 - \widetilde \Omega(t), 
\end{aligned}
\]
for some $C^*>0$, which, in addition to \eqref{eq:L^2 general0}, implies \eqref{eq:H^1 general0}.
\end{proof}

We end this section with the statement of our conclusion in this section. In addition to \eqref{theta(t)-2}, as an immediate corollary of Propositions \ref{prop:integrability of u general} and \ref{prop:integrability of u_x general} (via the standard limiting process), we have the following conclusion
\begin{corollary}
Under the same assumptions as in Propositions \ref{prop:integrability of u general} and \ref{prop:integrability of u_x general}, we have
\begin{equation}\label{eq:H^1 general}
\liminf_{t \to \infty}\int \sech^2 \left(\frac{x}{\lambda(t)}\right) \left(u^2 + u_x^2 \right) (t,x) \; dx =0.
\end{equation}
\end{corollary}

\subsubsection{CH and DP cases}
The proofs of \eqref{H decay CH} and \eqref{H decay DP} are simpler compared to BBM case, indeed, it suffices to consider one virial functional $\mathcal I_{\theta(t)}(t)$ where $\mathcal I(t)$ is given not in \eqref{eq:I}, but in \eqref{I}. Then, one has from Lemma \ref{lem:dt_I} and Remark \ref{rem:DP L1} that
\[\begin{aligned}
\frac{d}{dt} \mathcal I_{\theta(t)}(t) =&~{} -\frac{\theta'(t)}{(\theta(t))^2} \int \phi\left(\frac{x}{\lambda(t)}\right)u  -\frac{\la'(t)}{\theta(t)\la(t)} \int \frac{x}{\la(t)}\phi' \left(\frac x{\la(t)} \right)u \\
&~ {} +\frac1{\theta(t)\la(t)}  \int  \phi' \left(\frac x{\la(t)} \right)\left( \frac12u^2 + 3(1-\partial_x^2)^{-1} u^2\right),
\end{aligned}\]
for CH solutions, and 
\[\begin{aligned}
\frac{d}{dt} \mathcal I_{\theta(t)}(t) =&~{} -\frac{\theta'(t)}{(\theta(t))^2} \int \phi\left(\frac{x}{\lambda(t)}\right)u  -\frac{\la'(t)}{\theta(t)\la(t)} \int \frac{x}{\la(t)}\phi' \left(\frac x{\la(t)} \right)u \\
&~ {} +\frac1{2\theta(t)\la(t)}  \int  \phi' \left(\frac x{\la(t)} \right)\left( u^2 + 3(1-\partial_x^2)^{-1} u^2\right),
\end{aligned}\]
for DP solutions.

\medskip

We use \eqref{theta(t)-3} to control the first and second terms similarly as in the proof of Proposition \ref{prop:integrability of u general}. Then, from \eqref{exponential weight}, we have
\[\frac1{\theta(t)\la(t)}  \int  \phi' \left(\frac x{\la(t)} \right) \left(u^2 + u_x^2 \right)  \lesssim \Omega(t) + \frac{d}{dt} \mathcal I_{\theta(t)}(t),\]
for CH solutions, which implies (by passing to the standard argument) \eqref{H decay CH}. Similarly, we prove \eqref{H decay DP}.

\end{document}